\documentclass[11pt]{article}

\usepackage[T1]{fontenc}
\usepackage[utf8]{inputenc}
\usepackage{lmodern}
\usepackage{amsmath,amssymb,amsthm,mathtools}
\usepackage{microtype}
\usepackage[hidelinks]{hyperref}

\newtheorem{theorem}{Theorem}[section]
\newtheorem{proposition}[theorem]{Proposition}
\newtheorem{lemma}[theorem]{Lemma}
\newtheorem{corollary}[theorem]{Corollary}
\theoremstyle{definition}

\theoremstyle{remark}
\newtheorem{remark}[theorem]{Remark}

\DeclareMathOperator{\Tr}{Tr}

\DeclareMathOperator{\argmin}{arg\,min}

\DeclareMathOperator{\supp}{supp}
\DeclareMathOperator{\Lip}{Lip}

\newcommand{\B}{\mathbb{B}}
\newcommand{\C}{\mathbb{C}}
\newcommand{\R}{\mathbb{R}}
\newcommand{\Halpha}{\mathcal{H}_\alpha}
\newcommand{\nualpha}{\nu_\alpha}
\newcommand{\kalpha}[1]{k_{#1}^{\alpha}}
\newcommand{\Talpha}[1]{T_{#1}^{\alpha}}
\newcommand{\Ba}[1]{\widetilde{#1}^{\,\alpha}}

\newcommand{\eps}{\varepsilon}

\title{Orthogonality, Spectral Filling, and Uncertainty}
\author{Travis Alvarez \and Trevor Camper \and Mishko Mitkovski}
\date{}

\begin{document}

\maketitle

\begin{abstract}
Let a normalized continuous Parseval frame describe the phase space of a
reproducing kernel Hilbert space.  The frame coefficients of an orthonormal
family have total mass equal to the number of states and pointwise mass at
most one.  This elementary exclusion principle gives explicit lower bounds
for the total cost of a family in terms of the volume of low-cost regions.
For radial costs, it yields collective mean-dispersion, umbrella, and
transport inequalities.  The same argument applies to normalized Bessel
families, with the Bessel bound measuring the permitted overlap.

We also consider semiclassical families for which the frame measures vary.
Weak convergence of the normalized measures, together with averaged
localization of the normalized kernels, yields continuous-symbol Szeg\H{o}
limits and sharp asymptotics for the lowest Toeplitz energies.  The results
apply to Bargmann--Fock and weighted Bergman spaces, Gabor and
Paley--Wiener spaces, and orthogonal-polynomial and determinantal ensembles,
including the GUE.
\end{abstract}



\section{Introduction}\label{sec:introduction}

An orthonormal family cannot concentrate as if its members were independent.
In a phase-space representation, the coefficients of the whole family obey a
pointwise exclusion principle.  This observation is the source of the
uncertainty estimates in this paper.

To make the idea precise, let \((X,d)\) be a \(\sigma\)-compact locally
compact metric space, and let \(\mathcal H_\alpha\) be a reproducing kernel
Hilbert space with a normalized continuous Parseval frame
\(\{k_x^\alpha:x\in X\}\) and frame measure \(\nu_\alpha\).  Thus
\(\|k_x^\alpha\|_\alpha=1\) and
\[
    f=\int_X\langle f,k_x^\alpha\rangle_\alpha
                  k_x^\alpha\,d\nu_\alpha(x),
    \qquad f\in\mathcal H_\alpha,
\]
where the integral is weak.  If
\(\varphi_1,\ldots,\varphi_N\) is orthonormal, its occupation density is
\[
    h_{\boldsymbol\varphi}^\alpha(x)
       =\sum_{j=1}^N|\langle\varphi_j,k_x^\alpha\rangle_\alpha|^2.
\]
Parseval's identity and Bessel's inequality give
\[
    0\leq h_{\boldsymbol\varphi}^\alpha\leq1,
    \qquad
    \int_X h_{\boldsymbol\varphi}^\alpha\,d\nu_\alpha=N.
\]
Thus \(N\) orthogonal states occupy \(N\) units of frame measure, but no
point of phase space can be occupied by more than one unit.  Everything that
follows comes from combining this constraint with a cost function.

Hardy's theorem \cite{Hardy} is the classical model: simultaneous
localization in position and frequency is measured against a Gaussian scale,
and the harmonic oscillator turns those two costs into one energy.  Here the
role of the oscillator is played by a Toeplitz form, while the frame measure
itself determines how many low-cost states are available.  This makes the
argument applicable outside Euclidean phase space and, importantly, to
families in which the density of states changes with the parameter.

\subsection{Framework and main results}
\label{subsec:varying-frames}

The overlap measures
\[
    Q_\alpha(x,dy)
       =|\langle\kalpha{x},\kalpha{y}\rangle_\alpha|^2\,d\nualpha(y)
\]
are probability kernels.  For the semiclassical statements, let
\(c(\alpha)\to\infty\) and suppose that the normalized frame measures
\begin{equation}\label{eq:measure-convergence}
    dm_\alpha:=\frac{d\nualpha}{c(\alpha)}
       \rightharpoonup d\nu
\end{equation}
against \(C_c(X)\).  We also assume averaged localization: for every compact
\(K\subset X\) and every \(\eps>0\),
\begin{equation}\label{eq:averaged-localization}
    \int_KQ_\alpha\bigl(x,X\setminus B(x,\eps)\bigr)\,dm_\alpha(x)
       \longrightarrow0.
\end{equation}
These assumptions are weaker than exact proportionality
\(d\nu_\alpha=c(\alpha)d\nu\) and pointwise localization of every
normalized kernel.  The exact estimates below use the actual measure
\(\nu_\alpha\); the limiting measure enters only when \(\alpha\to\infty\).

\paragraph{Exact filling at one parameter.}
Let \(a:X\to[0,\infty)\) be a confining cost.  Write
\[
    F_{\alpha,a}(t)=\nualpha(\{a\leq t\}),\qquad
    F_{\alpha,a}^{\dagger}(s)
       =\inf\{t\geq0:F_{\alpha,a}(t)\geq s\}.
\]
Then every orthonormal family
\(\varphi_1,\ldots,\varphi_N\) satisfies the explicit bound
\[
    \sum_{j=1}^N\int_X a(x)
       |\langle\varphi_j,k_x^\alpha\rangle_\alpha|^2\,d\nu_\alpha(x)
       \geq\int_0^N F_{\alpha,a}^{\dagger}(s)\,ds.
\]
In words, the family must fill the lowest-cost \(N\) units of frame volume.
The same inequality holds for a positive trace-class operator
\(0\leq\gamma\leq I\) of trace \(N\), with the left-hand side replaced by
\[
    \int_Xa(x)\langle\gamma k_x^\alpha,k_x^\alpha\rangle_\alpha
       \,d\nu_\alpha(x).
\]
For a normalized Bessel family with Bessel bound \(B\), the right-hand side
becomes \(B\int_0^{N/B}F_{\alpha,a}^{\dagger}(s)\,ds\).  Thus the Bessel
bound records exactly how much local overlap is allowed.

For a set of finite frame volume, the associated concentration operator has
near-one eigenvectors that form the largest subspace of states concentrated
there.  When the overlap leaking across the boundary is lower order than the
volume, its dimension is asymptotic to that volume.  This is the
Landau--Slepian--Pollak aspect of the same Parseval-frame mechanism.

For a fixed \(e\in X\), let \(V_\alpha(r)=\nu_\alpha(B(e,r))\).  The
phase-space measure of a unit vector \(f\) is
\[
    d\mu_f^\alpha(x)
       =|\langle f,k_x^\alpha\rangle_\alpha|^2\,d\nu_\alpha(x).
\]
If \(\Delta_{\alpha,p}(f)\) denotes its least \(p\)-dispersion and
\(\mu_{\alpha,p}(f)\) a minimizing center, then
\[
    \sum_{j=1}^N
       \bigl[\Delta_{\alpha,p}(\varphi_j)
          +d(e,\mu_{\alpha,p}(\varphi_j))\bigr]^p
       \geq\int_0^N\bigl(V_\alpha^\dagger(s)\bigr)^p\,ds.
\]
This is an exact fixed-parameter uncertainty principle.  It has no
semiclassical hypothesis and no hidden spectral quantity.  The same filling
argument gives the umbrella, rearrangement, and Wasserstein estimates of
Section~\ref{sec:collective-uncertainty}.  In the Gabor model with a Gaussian
window, the case \(p=2\) reads
\[
 \sum_{j=1}^N\int_{\R^{2d}}(|x|^2+|\xi|^2)
          |V_g\varphi_j(x,\xi)|^2\,dx\,d\xi
 \geq\frac{d}{d+1}\frac{(d!)^{1/d}}{\pi}N^{1+1/d},
\]
which is the collective harmonic-oscillator form of Hardy uncertainty.

\paragraph{Semiclassical limits.}
For a bounded real symbol \(\sigma\), set
\[
    T_\sigma^\alpha f
       =\int_X\sigma(x)\langle f,k_x^\alpha\rangle_\alpha
                       k_x^\alpha\,d\nu_\alpha(x).
\]
Under \eqref{eq:measure-convergence}--\eqref{eq:averaged-localization},
if \(\eta\in C_c(X)\), \(\sigma\in C_b(X;\R)\), and \(\psi\) is
continuous on the range of \(\sigma\), then
\[
    \frac1{c(\alpha)}
       \Tr\bigl(T_\eta^\alpha\psi(T_\sigma^\alpha)\bigr)
       \longrightarrow\int_X\eta(x)\psi(\sigma(x))\,d\nu(x).
\]
This is a localized Szeg\H{o} limit; no commutation between the two Toeplitz
operators is required.

The same hypotheses identify the leading energy of the low Toeplitz
eigenstates.  Let \(\lambda_{1,\alpha}(a)\leq
\lambda_{2,\alpha}(a)\leq\cdots\) be the eigenvalues of the Toeplitz form
with confining symbol \(a\).  If \(N_\alpha/c(\alpha)\to\varrho\), then
weak convergence gives the lower bound
\[
    \liminf_{\alpha\to\infty}\frac1{c(\alpha)}
       \sum_{j=1}^{N_\alpha}\lambda_{j,\alpha}(a)
       \geq\int_0^\varrho F_a^\dagger(s)\,ds,
    \qquad F_a(t)=\nu(\{a\leq t\}).
\]
With averaged localization and the kernel-moment tail condition of
Section~\ref{sec:trace-limits}, this is an equality.  Thus the limiting
frame measure determines the leading cost of filling a macroscopic fraction
of the available states.

The exact occupation and concentration estimates are logically separate from
this limit theory: they use only the Parseval-frame identities at one
parameter.  Weak convergence and averaged localization enter only when a
semiclassical limit is asserted; unbounded costs require the additional tail
hypothesis stated later.

\subsection{Position in the literature}

Shapiro's umbrella theorem \cite{Shapiro} gives a qualitative starting
point: an orthonormal sequence cannot share integrable time and frequency
envelopes.  Jaming--Powell \cite{JamingPowell} and Malinnikova
\cite{Malinnikova} obtained quantitative cardinality and
mean-dispersion estimates for orthonormal sequences.  Gr\"ochenig--Zimmermann
\cite{GrochenigZimmermann} proved a Hardy theorem for the short-time Fourier
transform, while Gr\"ochenig--Malinnikova
\cite{GrochenigMalinnikova} established strong phase-space localization
obstructions for Riesz bases.  These results form the direct
time-frequency background for our collective estimates.  For broader
classical and dynamical Hardy theory, see Folland--Sitaram
\cite{FollandSitaram} and Fern\'andez-Bertolin--Malinnikova
\cite{FernandezBertolinMalinnikova}.

The filling principle is close to the distribution--energy inequalities of
Rumin \cite{Rumin}, and hence to the Berezin--Li--Yau and Lieb--Thirring
traditions.  Our contribution is to realize that principle with RKHS frame
measures, identify the corresponding radial Toeplitz spectra, and connect
the result directly to collective phase-space localization.

For the varying-frame trace theory, our point of departure is
Camper--Mitkovski \cite{CamperMitkovskiSzego}.  The present formulation
allows weakly varying diagonal frame measures and averaged, rather than
pointwise, kernel localization; this is what makes the orthogonal-polynomial
application fit the same framework.

\subsection{Applications and organization}

Section~\ref{sec:trace-limits} proves the varying-measure trace theorems,
the general concentration-operator dimension bound, and their unbounded
extensions.  Section~\ref{sec:radial-spectra} proves the exact filling
inequalities and their semiclassical limits.
Section~\ref{sec:collective-uncertainty} derives the mean-dispersion,
umbrella, rearrangement, and transport consequences.

Section~\ref{sec:applications} verifies the hypotheses and evaluates the
resulting costs in Bargmann--Fock, weighted Bergman, Gabor, and
Paley--Wiener spaces.  It then treats OPRL and \(\beta=2\) determinantal
ensembles, including the GUE.  The OPRL example is the main illustration of
why varying frame measures, rather than a fixed reference measure, are needed.

\section{Trace limits for varying frame measures}
\label{sec:trace-limits}

This section establishes the trace input for the uncertainty estimates.
We first record identities valid for each individual frame, and then use
\eqref{eq:measure-convergence} and \eqref{eq:averaged-localization} to pass
to the limit.  Finite total volume gives an unweighted Szeg\H{o} theorem.
On a noncompact space we obtain localized trace limits, followed by a
separate result for unbounded confining symbols.  The latter requires
control of kernel moments at infinity.

We use the notation of Section~\ref{subsec:varying-frames}.  The space
\(X\) is assumed \(\sigma\)-compact as well as locally compact and metric.
Thus the Radon measures are \(\sigma\)-finite, and continuity and the
Parseval identity make the frame spaces separable.  This convention
covers all the applications considered here.  Unless a statement is
explicitly made for a fixed parameter, both
\eqref{eq:measure-convergence} and \eqref{eq:averaged-localization} are
assumed.  We write \(C_b(X)\) for the bounded continuous functions and
\(\|\cdot\|_{\mathfrak S_1}\) for the trace norm.

\subsection{Frame identities and the Berezin transform}

The analysis map
\[
    \mathcal A_\alpha:\Halpha\longrightarrow L^2(X,\nualpha),
    \qquad
    (\mathcal A_\alpha f)(x)=\langle f,\kalpha{x}\rangle_\alpha
\]
is an isometry.  If \(\mathcal P_\alpha=\mathcal A_\alpha\mathcal
A_\alpha^*\) is the orthogonal projection onto its range, then for a
bounded symbol \(\sigma\),
\begin{equation}\label{eq:analysis-compression}
    \Talpha{\sigma}=\mathcal A_\alpha^*M_\sigma\mathcal A_\alpha.
\end{equation}
In particular, the Toeplitz map is positive and unital.

\begin{lemma}\label{lem:frame-calculus}
Fix \(\alpha\).  For a bounded symbol \(\sigma\),
\(\|\Talpha{\sigma}\|\leq\|\sigma\|_\infty\).
For \(\eta\in L^1(X,\nualpha)\), the rank-one integral defining
\(\Talpha{\eta}\) converges in trace norm and
\begin{equation}\label{eq:trace-symbol}
    \|\Talpha{\eta}\|_{\mathfrak S_1}
       \leq\int_X|\eta|\,d\nualpha,\qquad
    \Tr\Talpha{\eta}=\int_X\eta\,d\nualpha.
\end{equation}
More generally, for every bounded operator \(B\) on \(\Halpha\),
\begin{equation}\label{eq:weighted-trace-identity}
    \Tr(\Talpha{\eta}B)
       =\int_X\eta(x)\langle B\kalpha{x},\kalpha{x}\rangle_\alpha
           \,d\nualpha(x).
\end{equation}
The positive measure
\begin{equation}\label{eq:overlap-measure}
    d\Pi_\alpha(x,y):=dm_\alpha(x)\,Q_\alpha(x,dy)
      =\frac{|\langle\kalpha{x},\kalpha{y}\rangle_\alpha|^2}
              {c(\alpha)}\,d\nualpha(x)d\nualpha(y)
\end{equation}
is symmetric, and both its marginals are \(m_\alpha\).
Consequently the Berezin transform is a contraction on
\(L^p(X,m_\alpha)\) for \(1\leq p\leq\infty\), and preserves integrals
of \(L^1(m_\alpha)\) functions.
\end{lemma}

\begin{proof}
The norm bound follows from \eqref{eq:analysis-compression}.  The operator
\(\kalpha{x}\otimes\kalpha{x}\), defined by
\((u\otimes u)f=\langle f,u\rangle_\alpha u\), has trace norm and trace
equal to \(1\).  Integrating it against \(\eta\) proves
\eqref{eq:trace-symbol}; continuity of \(A\mapsto\Tr(AB)\) on
\(\mathfrak S_1\) proves \eqref{eq:weighted-trace-identity}.
Parseval applied to \(\kalpha{x}\) gives \(Q_\alpha(x,X)=1\).
Symmetry of \eqref{eq:overlap-measure} then identifies both marginals.
Jensen's inequality followed by the marginal identity gives the
\(L^p\) contraction for \(p<\infty\); the probability-kernel identity gives
the \(L^\infty\) bound.  Fubini gives preservation of integrals.
\end{proof}

For every positive bounded operator \(B\), Parseval and monotone
convergence also give
\begin{equation}\label{eq:positive-trace-frame}
    \Tr B=\int_X\langle B\kalpha{x},\kalpha{x}\rangle_\alpha
                       \,d\nualpha(x),
\end{equation}
with both sides allowed to be infinite.  Taking \(B=I\) yields
\(\dim\Halpha=\nualpha(X)\).  Thus finite frame mass is equivalent to
finite dimension.

For real bounded \(\sigma\), compression gives the operator inequality
\begin{equation}\label{eq:compression-variance}
\begin{aligned}
    \Talpha{\sigma^2}-(\Talpha{\sigma})^2
       &=\mathcal A_\alpha^*M_\sigma(I-\mathcal P_\alpha)
                    M_\sigma\mathcal A_\alpha\\
       &\geq0.
\end{aligned}
\end{equation}
If also \(\sigma\in L^1(\nualpha)\), all the following traces are finite,
and \eqref{eq:weighted-trace-identity} gives
\begin{equation}\label{eq:dirichlet-trace-defect}
\begin{aligned}
    \mathcal D_\alpha(\sigma)
       &:=\frac1{c(\alpha)}
            \Tr\bigl(\Talpha{\sigma^2}-(\Talpha{\sigma})^2\bigr)\\
       &=\int_X\sigma^2\,dm_\alpha
           -\int_X\sigma\,\Ba{\sigma}\,dm_\alpha\\
       &=\frac12\iint_{X\times X}
                   |\sigma(x)-\sigma(y)|^2\,d\Pi_\alpha(x,y).
\end{aligned}
\end{equation}
This is the normalized trace defect used in the introduction.

\subsection{Concentration operators and effective dimension}
\label{subsec:frame-concentration}

Let \(A\subset X\) be Borel with \(\nualpha(A)<\infty\).  The
positive contraction
\[
    C_{\alpha,A}:=\Talpha{\mathbf1_A}
\]
measures how much of a state is concentrated in \(A\).  Its eigenvectors
with eigenvalues close to one are the natural analogue of the prolate
functions.

\begin{theorem}
\label{thm:frame-concentration}
Fix \(\alpha\), let \(A\subset X\) be Borel with
\(\nualpha(A)<\infty\), and let \(0<\eps<1\).  Then \(C_{\alpha,A}\)
is a positive trace-class contraction and
\[
    \Tr C_{\alpha,A}=\nualpha(A),
    \qquad
    \Tr\bigl(C_{\alpha,A}-C_{\alpha,A}^2\bigr)
       =\int_AQ_\alpha(x,X\setminus A)\,d\nualpha(x).
\]
If
\[
    n_{\alpha,A}(\eps)
       =\#\{j:\lambda_j(C_{\alpha,A})\geq1-\eps\},
\]
where eigenvalues are counted with multiplicity, then
\[
\begin{aligned}
 \nualpha(A)
    -\frac1\eps\int_AQ_\alpha(x,X\setminus A)\,d\nualpha(x)
 &\leq n_{\alpha,A}(\eps)\\
 &\leq \nualpha(A)
    +\frac1{1-\eps}\int_AQ_\alpha(x,X\setminus A)\,d\nualpha(x).
\end{aligned}
\]
Moreover, \(n_{\alpha,A}(\eps)\) is the largest dimension of a
subspace \(E\subset\Halpha\) such that every unit vector \(f\in E\)
satisfies
\[
    \int_A|\langle f,k_x^\alpha\rangle_\alpha|^2\,d\nualpha(x)
       \geq1-\eps.
\]
In particular, every orthonormal family
\(\varphi_1,\ldots,\varphi_N\) for which each member has at least
\(1-\eps\) of its frame mass in \(A\) satisfies
\[
    N(1-\eps)\leq\nualpha(A).
\]
Consequently, for any sequence of finite-measure sets \(A_\alpha\), if
\[
    \int_{A_\alpha}Q_\alpha(x,X\setminus A_\alpha)
          \,d\nu_\alpha(x)
       =o\bigl(\nu_\alpha(A_\alpha)\bigr),
\]
then, for every fixed \(0<\eps<1\),
\[
    n_{\alpha,A_\alpha}(\eps)
       =\nu_\alpha(A_\alpha)+o\bigl(\nu_\alpha(A_\alpha)\bigr)
       \qquad(\alpha\to\infty).
\]
\end{theorem}

\begin{proof}
Lemma~\ref{lem:frame-calculus} gives \(0\leq C_{\alpha,A}\leq I\),
trace class, and \(\Tr C_{\alpha,A}=\nualpha(A)\).  Applying
\eqref{eq:weighted-trace-identity} with \(\eta=\mathbf1_A\) and
\(B=C_{\alpha,A}\) gives
\[
    \Tr C_{\alpha,A}^2
       =\int_AQ_\alpha(x,A)\,d\nualpha(x).
\]
The stated trace-defect formula follows because \(Q_\alpha(x,\cdot)\)
is a probability measure.

Write \(\lambda_j=\lambda_j(C_{\alpha,A})\), and separate the
eigenvalues according as \(\lambda_j\geq1-\eps\) or
\(\lambda_j<1-\eps\).  In the latter case,
\(\lambda_j\leq\eps^{-1}\lambda_j(1-\lambda_j)\), whereas in the
former,
\(1-\lambda_j\leq(1-\eps)^{-1}\lambda_j(1-\lambda_j)\).
Summing these two estimates proves the two bounds for
\(n_{\alpha,A}(\eps)\).  The spectral subspace for
\([1-\eps,1]\) has the asserted concentration property; the converse
dimension bound is the min--max principle.  For an orthonormal family,
sum the concentration inequalities and bound the resulting diagonal sum
of \(C_{\alpha,A}\) by its trace.  The last assertion follows immediately
from the two-sided estimate.
\end{proof}

The first trace identity is the familiar capacity estimate for an
orthonormal family.  The additional leakage term records the transition
between concentrated and unconcentrated states.  When it is lower order,
Theorem~\ref{thm:frame-concentration} is the Parseval-frame form of the
Landau--Slepian--Pollak effective-dimension principle
\cite{LandauPollak}.

\subsection{Averaged localization}

The convergence assumptions imply a local oscillation estimate for every
bounded continuous symbol.  Compact support converts it into global
Berezin convergence even when \(m_\alpha(X)=\infty\).

\begin{lemma}
\label{lem:averaged-berezin}
Let \(\sigma\in C_b(X)\), \(1\leq p<\infty\), and \(K\subset X\) be
compact.  Then
\begin{equation}\label{eq:local-oscillation}
    \int_K\int_X|\sigma(y)-\sigma(x)|^p
          \,Q_\alpha(x,dy)\,dm_\alpha(x)\longrightarrow0.
\end{equation}
In particular,
\(\int_K|\Ba{\sigma}-\sigma|^p\,dm_\alpha\to0\).
If \(\sigma\in C_c(X)\), then
\begin{equation}\label{eq:global-berezin-convergence}
    \|\Ba{\sigma}-\sigma\|_{L^p(m_\alpha)}\longrightarrow0.
\end{equation}
For real \(\sigma\in C_c(X)\), also
\(\mathcal D_\alpha(\sigma)\to0\).
\end{lemma}

\begin{proof}
A compactly supported continuous function that is at least \(1\) on \(K\),
together with \eqref{eq:measure-convergence}, bounds \(m_\alpha(K)\)
uniformly for all sufficiently large \(\alpha\).  Set
\[
    \omega_{\sigma,K}(\delta)
      =\sup\{|\sigma(y)-\sigma(x)|:x\in K,\ d(x,y)<\delta\}.
\]
Continuity and compactness imply
\(\omega_{\sigma,K}(\delta)\to0\) as \(\delta\downarrow0\).
The integral in \eqref{eq:local-oscillation} is bounded by
\[
    \omega_{\sigma,K}(\delta)^p m_\alpha(K)
      +(2\|\sigma\|_\infty)^p
        \int_K Q_\alpha(x,X\setminus B(x,\delta))\,dm_\alpha(x).
\]
First let \(\alpha\to\infty\), then \(\delta\downarrow0\).
Jensen's inequality gives the local Berezin assertion.

If \(\sigma\) is compactly supported, take \(K=\supp\sigma\).
The difference \(\sigma(y)-\sigma(x)\) vanishes unless \(x\in K\) or
\(y\in K\).  Symmetry of \(\Pi_\alpha\) therefore bounds its global
\(p\)-th oscillation integral by twice the integral in
\eqref{eq:local-oscillation}.  Jensen gives
\eqref{eq:global-berezin-convergence}, and
\eqref{eq:dirichlet-trace-defect} proves the last assertion.
\end{proof}

\subsection{Berezin--Lieb inequalities and finite total volume}

We recall the finite-mass form of the Berezin--Lieb inequalities
\cite{BerezinSymbols,LiebClassical,CamperMitkovskiSzego}.  This is a
fixed-parameter statement; no convergence hypothesis is used.

\begin{proposition}
\label{prop:berezin-lieb}
Suppose \(\nualpha(X)<\infty\).  Let \(\sigma\) be a bounded real Borel
function taking values in a compact interval \(I\), and let
\(\psi:I\to\R\) be continuous and convex.  Then
\begin{equation}\label{eq:berezin-lieb}
    \int_X\psi(\Ba{\sigma})\,dm_\alpha
       \leq\frac1{c(\alpha)}\Tr\psi(\Talpha{\sigma})
       \leq\int_X\psi(\sigma)\,dm_\alpha.
\end{equation}
\end{proposition}

\begin{proof}
The space is finite-dimensional.  Choose an orthonormal eigenbasis \(e_j\)
of \(\Talpha{\sigma}\), with eigenvalues \(\lambda_j\in I\).
For each \(x\), the numbers
\(|\langle e_j,\kalpha{x}\rangle_\alpha|^2\) sum to \(1\), and
\[
    \Ba{\sigma}(x)
       =\sum_j\lambda_j|\langle e_j,\kalpha{x}\rangle_\alpha|^2.
\]
Jensen, integration, and Parseval prove the left inequality.
For each \(j\), the measure
\(|\langle e_j,\kalpha{x}\rangle_\alpha|^2\,d\nualpha(x)\)
is a probability measure, with \(\sigma\)-average \(\lambda_j\).
Applying Jensen to it and summing over \(j\) proves the right inequality.
\end{proof}

\begin{theorem}
\label{thm:finite-volume-szego}
Suppose \(\nualpha(X)<\infty\) for every \(\alpha\), and, in addition to
\eqref{eq:measure-convergence} and \eqref{eq:averaged-localization}, assume
\begin{equation}\label{eq:total-mass-convergence}
    m_\alpha(X)\longrightarrow\nu(X)<\infty.
\end{equation}
Let \(\sigma\in C_b(X;\R)\), and let \(I\) be a compact interval
containing its range.  For every \(\psi\in C(I)\),
\begin{equation}\label{eq:finite-volume-szego}
    \frac1{c(\alpha)}\Tr\psi(\Talpha{\sigma})
       \longrightarrow\int_X\psi(\sigma)\,d\nu.
\end{equation}
In particular, this holds on compact \(X\) under the two standing
convergence assumptions alone.
\end{theorem}

\begin{proof}
Vague convergence and \eqref{eq:total-mass-convergence} imply asymptotic
tightness and convergence against \(C_b(X)\).  Indeed, choose
\(0\leq\chi\leq1\) in \(C_c(X)\) with
\(\int_X(1-\chi)\,d\nu\) small, and use
\[
    \int_X(1-\chi)\,dm_\alpha
       =m_\alpha(X)-\int_X\chi\,dm_\alpha
       \longrightarrow\int_X(1-\chi)\,d\nu.
\]
Lemma~\ref{lem:averaged-berezin} on \(\supp\chi\), together with
\(|\Ba{\sigma}-\sigma|\leq2\|\sigma\|_\infty\), now gives
\(\|\Ba{\sigma}-\sigma\|_{L^1(m_\alpha)}\to0\).
Uniform continuity of \(\psi\) and boundedness of \(m_\alpha(X)\) imply
\[
    \int_X\bigl|\psi(\Ba{\sigma})-\psi(\sigma)\bigr|\,dm_\alpha
       \longrightarrow0.
\]
For convex \(\psi\), \eqref{eq:berezin-lieb} and convergence against
\(C_b(X)\) prove \eqref{eq:finite-volume-szego}.
Every real continuous function on \(I\) is uniformly approximable by
piecewise linear functions, each a difference of two continuous convex
functions.  The normalized trace error of an approximation of size
\(\eps\) is at most
\(\eps\,\dim\Halpha/c(\alpha)=\eps m_\alpha(X)\).
This proves the result for real \(\psi\), and then for complex \(\psi\)
by linearity.  On compact \(X\), the test function \(1\) gives
\eqref{eq:total-mass-convergence}.
\end{proof}

The mass condition matters on noncompact spaces: vague convergence alone
allows normalized dimension to escape to infinity.  Theorem
\ref{thm:finite-volume-szego} therefore applies both to compact phase
spaces and to finite-dimensional polynomial ensembles on \(\R\) whose
normalized diagonal measures converge weakly as probabilities.

\subsection{Localized traces without commutation}

A compactly supported weight removes the need for finite total frame
mass.  We use the small spectral variance of a coherent state to pass
from the symbol to continuous functions of its Toeplitz operator.

\begin{theorem}
\label{thm:localized-szego}
Let \(\sigma\in C_b(X;\R)\), let \(I\) be a compact interval containing
its range, and let \(\psi\in C(I)\).  For every compact \(K\subset X\),
\begin{equation}\label{eq:localized-functional-calculus}
    \int_K
      \left|\langle\psi(\Talpha{\sigma})\kalpha{x},\kalpha{x}
                    \rangle_\alpha-\psi(\sigma(x))\right|
                  \,dm_\alpha(x)\longrightarrow0.
\end{equation}
Consequently, for every \(\eta\in C_c(X)\),
\begin{equation}\label{eq:localized-szego}
    \frac1{c(\alpha)}
       \Tr\bigl(\Talpha{\eta}\psi(\Talpha{\sigma})\bigr)
       \longrightarrow\int_X\eta(x)\psi(\sigma(x))\,d\nu(x).
\end{equation}
The operators \(\Talpha{\eta}\) and \(\Talpha{\sigma}\) need not commute.
\end{theorem}

\begin{proof}
Write \(A=\Talpha{\sigma}\).  By \eqref{eq:compression-variance},
\[
\begin{aligned}
    v_\alpha(x)
       &:=\|(A-\sigma(x)I)\kalpha{x}\|_\alpha^2\\
       &\leq\int_X|\sigma(y)-\sigma(x)|^2\,Q_\alpha(x,dy).
\end{aligned}
\]
Let \(\omega_\psi(\delta)\) be the modulus of continuity of \(\psi\)
on \(I\).  The spectral probability measure of \(A\) in the vector
\(\kalpha{x}\) assigns mass at most \(v_\alpha(x)/\delta^2\) to
\(\{t:|t-\sigma(x)|>\delta\}\).  Hence
\[
    \left|\langle\psi(A)\kalpha{x},\kalpha{x}\rangle_\alpha
                      -\psi(\sigma(x))\right|
       \leq\omega_\psi(\delta)
          +\frac{2\|\psi\|_\infty}{\delta^2}\,v_\alpha(x).
\]
Integrate over \(K\).  Its \(m_\alpha\)-mass stays bounded, while
Lemma~\ref{lem:averaged-berezin} gives
\(\int_Kv_\alpha\,dm_\alpha\to0\).
Letting first \(\alpha\to\infty\), then \(\delta\downarrow0\), proves
\eqref{eq:localized-functional-calculus}.
Formula \eqref{eq:weighted-trace-identity} identifies the left-hand side
of \eqref{eq:localized-szego} with the integral of
\(\eta(x)\langle\psi(A)\kalpha{x},\kalpha{x}\rangle_\alpha\)
against \(m_\alpha\).  The preceding convergence and vague convergence
applied to \(\eta\psi(\sigma)\in C_c(X)\) finish the proof.
\end{proof}

This proof gives a localized trace limit for continuous symbols without
the commutation hypothesis used in the weighted Berezin--Lieb argument
of \cite{CamperMitkovskiSzego}.  It also gives an unweighted trace-class
specialization on noncompact spaces.

\begin{corollary}\label{cor:trace-class-szego}
Let \(\sigma\in C_c(X;\R)\), and let \(I\) be a compact interval
containing \(0\) and the range of \(\sigma\).  Suppose
\(\psi(t)=t\,g(t)\) on \(I\), with \(g\in C(I)\).  Then
\(\psi(\Talpha{\sigma})\) is trace class and
\[
    \frac1{c(\alpha)}\Tr\psi(\Talpha{\sigma})
       \longrightarrow\int_X\psi(\sigma)\,d\nu.
\]
\end{corollary}

\begin{proof}
Lemma~\ref{lem:frame-calculus} makes \(\Talpha{\sigma}\) trace class, and
\(\psi(\Talpha{\sigma})=\Talpha{\sigma}g(\Talpha{\sigma})\).
Apply Theorem~\ref{thm:localized-szego} with weight \(\eta=\sigma\)
and function \(g\).
\end{proof}

The condition at zero ensures trace-class functional calculus even if
\(\Halpha\) is infinite-dimensional.  It includes powers \(t^k\),
\(k\geq1\), and functions \((t-s)_+\) with \(s>0\).

\subsection{Which measurable symbols are allowed?}

Weak convergence controls continuous test functions.  The following
criterion records the additional approximation needed to retain rough
symbols.

\begin{corollary}
\label{cor:rough-symbols}
Assume the hypotheses of Theorem~\ref{thm:finite-volume-szego}.
Let \(\sigma\) be a bounded real Borel function.  Suppose there are real
\(\sigma_j\in C_b(X)\), uniformly bounded together with \(\sigma\), such that
\begin{equation}\label{eq:rough-symbol-approximation}
    \lim_{j\to\infty}
      \left(\limsup_{\alpha\to\infty}
                \|\sigma-\sigma_j\|_{L^2(m_\alpha)}
                +\|\sigma-\sigma_j\|_{L^2(\nu)}\right)=0.
\end{equation}
Then \eqref{eq:finite-volume-szego} holds for \(\sigma\) and every
continuous \(\psi\) on a compact interval containing all these ranges.
\end{corollary}

\begin{proof}
The proof of Theorem~\ref{thm:finite-volume-szego} and boundedness of the
symbols give
\(\|\widetilde{\sigma_j}^{\,\alpha}-\sigma_j\|_{L^2(m_\alpha)}\to0\)
for each \(j\).  The contraction property implies
\[
    \|\Ba{\sigma}-\sigma\|_{L^2(m_\alpha)}
       \leq2\|\sigma-\sigma_j\|_{L^2(m_\alpha)}
          +\|\widetilde{\sigma_j}^{\,\alpha}-\sigma_j\|_{L^2(m_\alpha)}.
\]
Thus the left-hand side tends to zero.  Uniform continuity of \(\psi\),
bounded total masses, and \eqref{eq:rough-symbol-approximation} also give
\(\int\psi(\sigma)\,dm_\alpha\to\int\psi(\sigma)\,d\nu\),
by first using \(\sigma_j\) and then letting \(j\to\infty\).
The Berezin--Lieb proof now applies unchanged.
\end{proof}

When \(m_\alpha=\nu\) is finite, density of continuous functions in
\(L^2(\nu)\), followed by truncation, supplies
\eqref{eq:rough-symbol-approximation} for every bounded real Borel symbol.
Thus the earlier bounded-measurable result is recovered in that case.

\begin{remark}\label{rem:rough-symbol-counterexample}
The approximation condition cannot be omitted.  On \(X=[0,1]\), take
\(\mathcal H_n=\mathcal P_{n-1}\) with inner product
\(\langle f,g\rangle=\sum_{j=1}^n f(j/n)\overline{g(j/n)}\).
The cardinal Lagrange polynomials form an orthonormal basis; the
normalized reproducing kernels are continuous, and the frame measure is
\(\nu_n=\sum_{j=1}^n\delta_{j/n}\).
The kernels at the nodes are orthonormal, so averaged off-diagonal
leakage is identically zero.  With \(c(n)=n\), the measures \(m_n\)
converge weakly to Lebesgue measure.  Nevertheless, for
\(\sigma=\mathbf1_{\mathbb Q\cap[0,1]}\), one has
\(T_\sigma^n=I\), hence \(n^{-1}\Tr T_\sigma^n=1\), whereas
\(\int_0^1\sigma(x)\,dx=0\).
Thus symbols must be interpreted as Borel functions with specified
representatives when the frame measures have different null sets.
\end{remark}

\subsection{Confining symbols and Riesz means}

Let \(a:X\to[0,\infty)\) be continuous with compact sublevel sets.  For
each \(\alpha\), assume that the form
\begin{equation}\label{eq:unbounded-toeplitz-form}
    q_{\alpha,a}(f)
       =\int_X a(x)|\langle f,\kalpha{x}\rangle_\alpha|^2
                         \,d\nualpha(x)
\end{equation}
is densely defined on \(\Halpha\).
It is closed, since \(M_{\sqrt a}\mathcal A_\alpha\) is a closed
operator.  Denote its associated nonnegative self-adjoint operator by
\(T_a^\alpha\), and set
\(\widetilde a^{\,\alpha}(x)=\int_Xa(y)\,Q_\alpha(x,dy)\), allowing
the value \(+\infty\).  For \(t>0\), write
\[
    R_{\alpha,a}(t)=\Tr(t-T_a^\alpha)_+,
    \qquad s_+=\max\{s,0\}.
\]

\begin{proposition}\label{prop:riesz-sandwich}
For every \(\alpha\) and \(t>0\),
\begin{equation}\label{eq:riesz-sandwich}
    \int_X(t-\widetilde a^{\,\alpha})_+\,dm_\alpha
       \leq\frac{R_{\alpha,a}(t)}{c(\alpha)}
       \leq\int_X(t-a)_+\,dm_\alpha<\infty,
\end{equation}
where \((t-\infty)_+=0\).  In particular, \(T_a^\alpha\) has compact
resolvent.  Neither measure convergence nor averaged localization is
needed for these assertions.
\end{proposition}

\begin{proof}
For any finite orthonormal family \(u_1,\ldots,u_N\) in the form domain,
put \(s(x)=\sum_j|\langle u_j,\kalpha{x}\rangle_\alpha|^2\).
Then \(0\leq s\leq1\), \(\int s\,d\nualpha=N\), and
\[
    \sum_{j=1}^N(t-q_{\alpha,a}(u_j))
       =\int_X(t-a)s\,d\nualpha
       \leq\int_X(t-a)_+\,d\nualpha.
\]
The variational formula for the trace of a positive part proves the
upper bound.  The last integral is finite because \((t-a)_+\) is
continuous and compactly supported.

For \(x\) with \(\widetilde a^{\,\alpha}(x)<\infty\), the vector
\(\kalpha{x}\) belongs to the form domain and has energy
\(q_{\alpha,a}(\kalpha{x})=\widetilde a^{\,\alpha}(x)\).
Jensen's inequality for its spectral probability measure gives
\[
    (t-\widetilde a^{\,\alpha}(x))_+
       \leq\langle(t-T_a^\alpha)_+\kalpha{x},\kalpha{x}\rangle_\alpha.
\]
The same inequality is trivial when the left side is defined to be zero.
Integrate and use \eqref{eq:positive-trace-frame}.
Finally, for \(r\geq0\),
\(\mathbf1_{[0,r]}(T_a^\alpha)\leq(r+1-T_a^\alpha)_+\).
All bounded spectral intervals therefore have finite rank, which proves
compactness of the resolvent.
\end{proof}

To obtain matching asymptotics, assume the following local control of the
unbounded kernel moment: for every compact \(K\subset X\),
\begin{equation}\label{eq:kernel-moment-tail}
    \lim_{R\to\infty}\limsup_{\alpha\to\infty}
       \int_K\int_{\{y:a(y)>R\}}
          a(y)\,Q_\alpha(x,dy)\,dm_\alpha(x)=0.
\end{equation}
This condition is automatic for bounded \(a\).  For an unbounded cost it
is a separate hypothesis; averaged localization alone controls only
bounded oscillations.

\begin{lemma}
\label{lem:unbounded-berezin}
Under \eqref{eq:kernel-moment-tail},
\[
    \int_K|\widetilde a^{\,\alpha}-a|\,dm_\alpha\longrightarrow0
\]
for every compact \(K\subset X\).
\end{lemma}

\begin{proof}
Let \(a_R=\min\{a,R\}\in C_b(X)\), and take \(R>\sup_Ka\).
Then
\[
    \int_K|\widetilde a^{\,\alpha}-a|\,dm_\alpha
       \leq\int_K|\widetilde{a_R}^{\,\alpha}-a_R|\,dm_\alpha
          +\int_K\int_{\{a(y)>R\}}a(y)\,
                       Q_\alpha(x,dy)\,dm_\alpha(x).
\]
The first term tends to zero by Lemma~\ref{lem:averaged-berezin}.
Take the upper limit in \(\alpha\), then let \(R\to\infty\).
\end{proof}

\begin{theorem}\label{thm:riesz-limit}
Let \(a:X\to[0,\infty)\) be continuous with compact sublevel sets.
Assume the standing convergence hypotheses, density of the domains of
the forms \eqref{eq:unbounded-toeplitz-form}, and
\eqref{eq:kernel-moment-tail}.  Then, for every \(t>0\),
\begin{equation}\label{eq:riesz-limit}
    \frac1{c(\alpha)}\Tr(t-T_a^\alpha)_+
       \longrightarrow\int_X(t-a(x))_+\,d\nu(x).
\end{equation}
\end{theorem}

\begin{proof}
The upper bound in \eqref{eq:riesz-sandwich} converges to the proposed
limit by vague convergence.  Choose \(0\leq\chi\leq1\) in \(C_c(X)\),
equal to \(1\) on the compact set \(\{a\leq t\}\).
The lower bound is at least
\(\int_X\chi(t-\widetilde a^{\,\alpha})_+\,dm_\alpha\).
Since \(s\mapsto(t-s)_+\) is \(1\)-Lipschitz,
Lemma~\ref{lem:unbounded-berezin} shows that this differs from
\(\int_X\chi(t-a)_+\,dm_\alpha\) by a quantity tending to zero.
The latter integral has the required limit.
\end{proof}

\begin{corollary}\label{cor:counting-limit}
Under the hypotheses of Theorem~\ref{thm:riesz-limit}, if \(t>0\) and
\(\nu(\{a=t\})=0\), then
\[
    \frac1{c(\alpha)}
       \Tr\mathbf1_{[0,t]}(T_a^\alpha)
       \longrightarrow\nu(\{a\leq t\}).
\]
\end{corollary}

\begin{proof}
For \(0<h<t\), the spectral theorem gives
\[
    \frac{R_{\alpha,a}(t)-R_{\alpha,a}(t-h)}h
       \leq\Tr\mathbf1_{[0,t]}(T_a^\alpha)
       \leq\frac{R_{\alpha,a}(t+h)-R_{\alpha,a}(t)}h.
\]
Divide by \(c(\alpha)\), apply \eqref{eq:riesz-limit}, and let
\(h\downarrow0\).  The assumption on the level set makes both limiting
difference quotients equal to \(\nu(\{a\leq t\})\).
\end{proof}

The upper Riesz-mean bound in Proposition~\ref{prop:riesz-sandwich}
already yields spectral lower bounds without a limit.
Theorem~\ref{thm:riesz-limit} identifies the asymptotic filling profile
when the additional moment condition holds.  These are the two forms of
the trace input used for radial costs in the next section.

% -----------------------------------------------------------------------------
\section{Radial spectra and occupation bounds}
\label{sec:radial-spectra}

The trace bounds of the preceding section measure the supply of low-energy
states.  We now turn them into a price for occupying those states.
At a fixed parameter the price is bounded below by the cost of filling
the cheapest available frame volume.  Kernel moment estimates give a
matching upper bound, while the semiclassical hypotheses identify the
limiting cost.  These statements apply to orthonormal families and, more
generally, to positive trace-class operators with occupation numbers at
most one.

All results in this section up to
Subsection~\ref{subsec:semiclassical-filling} concern a fixed frame and
require no convergence assumptions.  We retain the subscript \(\alpha\)
to keep track of the actual measure \(\nualpha\).

\subsection{Radial forms and their lowest eigenvalue}

For radial symbols, assume that \(X\) is proper, so that closed bounded
sets are compact.  Fix \(e\in X\), put \(d_e(x)=d(x,e)\), and let
\(\rho:[0,\infty)\to[0,\infty)\) be continuous, strictly increasing,
with \(\rho(0)=0\) and \(\rho(r)\to\infty\).
Then
\[
    a(x)=\rho(d_e(x))
\]
has compact sublevel sets.  We assume that its quadratic form
\(q_{\alpha,a}\), defined in \eqref{eq:unbounded-toeplitz-form}, is
densely defined.  A sufficient condition is
\(\widetilde a^{\,\alpha}(x)<\infty\) for \(\nualpha\)-almost every
\(x\): those frame vectors lie in the form domain and have dense span
by Parseval.  This is a sufficient test, not an additional standing
assumption.  As in the introduction, \(T_d^\alpha\) denotes the choice
\(a=d_e\); we write \(T_{d^p}^\alpha\) when \(a=d_e^p\).

\begin{proposition}
\label{prop:radial-spectrum}
The operator \(T_a^\alpha\) has compact resolvent.  Its eigenvalues,
repeated with multiplicity, can therefore be ordered as
\[
    0\leq\lambda_{1,\alpha}(a)\leq\lambda_{2,\alpha}(a)
       \leq\cdots,
\]
with \(\lambda_{j,\alpha}(a)\to\infty\) if \(\Halpha\) is
infinite dimensional.  Moreover, \(\nualpha(\{e\})\leq1\).
If \(\nualpha(\{e\})<1\), then
\[
    \lambda_{1,\alpha}(a)
       \geq\sup_{r>0}\rho(r)\bigl(1-V_\alpha(r)\bigr)_+>0,
    \qquad V_\alpha(r)=\nualpha(B(e,r)),
\]
and \(T_a^\alpha\) has compact inverse.  If
\(\nualpha(\{e\})=1\), its kernel is exactly
\(\operatorname{span}\{\kalpha{e}\}\).
\end{proposition}

\begin{proof}
Compactness of the resolvent follows from
Proposition~\ref{prop:riesz-sandwich}.  Parseval applied to
\(\kalpha{e}\) gives \(\nualpha(\{e\})\leq1\).
For any unit vector \(f\) in the form domain, the pointwise bound
\(|\langle f,\kalpha{x}\rangle_\alpha|^2\leq1\) implies
\[
    q_{\alpha,a}(f)
       \geq\rho(r)\int_{X\setminus B(e,r)}
                    |\langle f,\kalpha{x}\rangle_\alpha|^2\,d\nualpha(x)
       \geq\rho(r)\bigl(1-V_\alpha(r)\bigr).
\]
The bound is positive for some small \(r>0\) when
\(\nualpha(\{e\})<1\), since
\(V_\alpha(r)\downarrow\nualpha(\{e\})\) as \(r\downarrow0\).

If \(\nualpha(\{e\})=1\), Parseval shows that
\(\langle\kalpha{e},\kalpha{x}\rangle_\alpha=0\) for almost every
\(x\ne e\), and hence \(q_{\alpha,a}(\kalpha{e})=0\).
Conversely, a vector of zero energy has frame coefficients supported
on \(\{e\}\).  The reconstruction formula then makes it a multiple
of \(\kalpha{e}\).
\end{proof}

The compactness mechanism can also be seen directly in the estimate
\[
    \int_{X\setminus B(e,R)}
         |\langle f,\kalpha{x}\rangle_\alpha|^2\,d\nualpha(x)
       \leq\frac{q_{\alpha,a}(f)}{\rho(R)}.
\]
Restriction of the analysis map to a fixed ball is Hilbert--Schmidt,
with squared Hilbert--Schmidt norm \(V_\alpha(R)\).  A bounded set
with bounded radial energy is therefore precompact.  This is the
radial instance of the frame compactness approach of
Mitkovski--Stockdale--Wagner--Wick \cite{MSWWCompactness}.
The distinction between compact resolvent and compact inverse is
essential when the frame measure has atoms.

\subsection{Filling volume, including atoms}
\label{subsec:classical-filling}

For a general continuous nonnegative confining symbol \(a\), define
\[
\begin{aligned}
    F_{\alpha,a}(t)&=\nualpha(\{a\leq t\}),\qquad t\geq0,\\
    F_{\alpha,a}^{\dagger}(s)
       &=\inf\{t\geq0:F_{\alpha,a}(t)\geq s\},\\
    \Phi_{\alpha,a}(N)&=\int_0^N F_{\alpha,a}^{\dagger}(s)\,ds.
\end{aligned}
\]
We use \(\inf\varnothing=+\infty\) and
\(\Phi_{\alpha,a}(0)=0\).  The function \(\Phi_{\alpha,a}\) is the
least classical cost of placing mass \(N\) with density at most one.
Allowing fractional occupation of a level set accommodates atoms and
removes any need to differentiate the volume function.

\begin{lemma}\label{lem:filling-duality}
For every finite \(0\leq N\leq\nualpha(X)\),
\[
\begin{aligned}
    \Phi_{\alpha,a}(N)
       &=\inf_{\substack{0\leq h\leq1\\ \int_Xh\,d\nualpha=N}}
                 \int_X a h\,d\nualpha\\
       &=\int_0^\infty\bigl(N-F_{\alpha,a}(t)\bigr)_+\,dt\\
       &=\sup_{t\geq0}
            \left\{Nt-\int_X(t-a)_+\,d\nualpha\right\}.
\end{aligned}
\]
The infimum is attained.  For \(0<N<\nualpha(X)\), one may take
\begin{equation}\label{eq:classical-filling-density}
    h_N=\mathbf1_{\{a<\tau\}}+\theta\mathbf1_{\{a=\tau\}},
    \qquad \tau=F_{\alpha,a}^{\dagger}(N),\quad 0\leq\theta\leq1,
\end{equation}
where \(\theta\) is chosen to make the total mass \(N\).
\end{lemma}

\begin{proof}
If \(h\) is admissible, layer cake and \(h\leq1\) give
\[
    \int_Xah\,d\nualpha
       =\int_0^\infty\int_{\{a>t\}}h\,d\nualpha\,dt
       \geq\int_0^\infty\bigl(N-F_{\alpha,a}(t)\bigr)_+\,dt.
\]
Tonelli's theorem applied to the generalized inverse identifies the
last expression with \(\int_0^NF_{\alpha,a}^{\dagger}(s)\,ds\).
For an interior mass \(N\), right continuity of \(F_{\alpha,a}\)
and finiteness of its values imply
\[
    \nualpha(\{a<\tau\})\leq N\leq\nualpha(\{a\leq\tau\}).
\]
Thus \eqref{eq:classical-filling-density} is admissible and attains
the bound.  At the endpoints use \(h=0\), or \(h=1\) if
\(N=\nualpha(X)<\infty\).

Finally, layer cake gives
\[
    \int_X(t-a)_+\,d\nualpha=\int_0^tF_{\alpha,a}(u)\,du.
\]
Since \(N-F_{\alpha,a}(u)\) is nonincreasing, the supremum of its
integral over \([0,t]\) is the integral of its positive part over
\([0,\infty)\).  This proves the last equality, including cases
where the value is infinite.
\end{proof}

\subsection{Occupation bounds and spectral filling}
\label{subsec:occupation-spectral-filling}

Assume that \(q_{\alpha,a}\) is densely defined, and let
\(D_\alpha=\dim\Halpha=\nualpha(X)\), allowing \(D_\alpha=\infty\).
For real \(0\leq N\leq D_\alpha\), define the interpolated
eigenvalue sum
\begin{equation}\label{eq:spectral-filling-function}
    S_{\alpha,a}(N)
       =\sum_{j=1}^{n}\lambda_{j,\alpha}(a)
          +\theta\lambda_{n+1,\alpha}(a),
    \qquad N=n+\theta.
\end{equation}
Here \(n\in\mathbb N\cup\{0\}\) and \(0\leq\theta<1\).
The last term is omitted when \(\theta=0\).  Thus the definition
also applies at full occupation of a finite-dimensional space.

For \(0\leq\gamma\leq I\) of trace class, set
\[
    h_\gamma^\alpha(x)
       =\langle\gamma\kalpha{x},\kalpha{x}\rangle_\alpha,
    \qquad
    \mathcal E_{\alpha,a}(\gamma)
       =\int_Xa(x)h_\gamma^\alpha(x)\,d\nualpha(x).
\]
This is the extended quadratic-form trace of \(\gamma T_a^\alpha\).
Equivalently, if \(\gamma u_j=\theta_j u_j\) with \(\theta_j>0\),
then \(\mathcal E_{\alpha,a}(\gamma)=\sum_j\theta_jq_{\alpha,a}(u_j)\),
with value \(+\infty\) allowed.
We also recall the scalar filling cost
\(\Phi_{\alpha,a}(N)=\int_0^N F_{\alpha,a}^{\dagger}(s)\,ds\).

\begin{theorem}
\label{thm:occupation-filling}
Let \(0\leq\gamma\leq I\) be trace class, with \(N=\Tr\gamma\).  Then
\begin{equation}\label{eq:occupation-filling-chain}
    \mathcal E_{\alpha,a}(\gamma)
       \geq S_{\alpha,a}(N)
       \geq \Phi_{\alpha,a}(N).
\end{equation}
More precisely,
\begin{equation}\label{eq:spectral-filling-duality}
\begin{aligned}
    S_{\alpha,a}(N)
       &=\min_{\substack{0\leq\eta\leq I\\\Tr\eta=N}}
          \mathcal E_{\alpha,a}(\eta)\\
       &=\sup_{t\geq0}\{Nt-\Tr(t-T_a^\alpha)_+\}.
\end{aligned}
\end{equation}
For integer \(N\), the minimum is attained by projection onto the
first \(N\) eigenvectors.  In particular, every orthonormal family
\(\varphi_1,\ldots,\varphi_N\) satisfies
\begin{equation}\label{eq:orthonormal-filling-bound}
    \sum_{j=1}^Nq_{\alpha,a}(\varphi_j)
       \geq S_{\alpha,a}(N)
       \geq \Phi_{\alpha,a}(N).
\end{equation}
\end{theorem}

\begin{proof}
The operator inequality \(0\leq\gamma\leq I\) gives
\(0\leq h_\gamma^\alpha\leq1\).  Parseval and Tonelli give
\(\int_Xh_\gamma^\alpha\,d\nualpha=\Tr\gamma=N\), so
Lemma~\ref{lem:filling-duality} implies
\(\mathcal E_{\alpha,a}(\gamma)\geq\Phi_{\alpha,a}(N)\).

Let \(v_j\) be an orthonormal eigenbasis of \(T_a^\alpha\), and put
\(b_j=\langle\gamma v_j,v_j\rangle_\alpha\).  Then
\[
    0\leq b_j\leq1,\qquad \sum_jb_j=N,\qquad
    \mathcal E_{\alpha,a}(\gamma)
       =\sum_j\lambda_{j,\alpha}(a)b_j.
\]
Filling this nondecreasing sequence from the bottom minimizes the last
sum.  Equivalently, apply Lemma~\ref{lem:filling-duality} to counting
measure on the eigenvalue indices.  The minimizer has \(b_j=1\)
for \(j\leq n\), \(b_{n+1}=\theta\), and all other entries zero.
Its cost is \eqref{eq:spectral-filling-function}, and the same lemma
gives the Riesz-mean duality \eqref{eq:spectral-filling-duality}.
Applying the classical lower bound to this minimizing operator proves
\(S_{\alpha,a}(N)\geq\Phi_{\alpha,a}(N)\).
Taking \(\gamma\) to be the projection onto the given orthonormal
family proves \eqref{eq:orthonormal-filling-bound}.
\end{proof}

This is the frame realization of the exclusion and distribution--energy
mechanism discussed in \cite{Rumin}.  It uses the actual frame measure
at each parameter, including its atoms and its nonconstant density.
The trace upper bound in Proposition~\ref{prop:riesz-sandwich} gives
the same lower estimate by convex duality.

\begin{corollary}\label{cor:radial-filling}
Under the radial assumptions above, let
\[
    V_\alpha^{\dagger}(s)
       =\inf\{r\geq0:V_\alpha(r)\geq s\}.
\]
If \(N=n+\theta\), where \(n\in\mathbb N\cup\{0\}\) and
\(0\leq\theta<1\), then
\begin{equation}\label{eq:radial-spectral-filling}
    \sum_{j=1}^n\lambda_{j,\alpha}(\rho\circ d_e)
       +\theta\lambda_{n+1,\alpha}(\rho\circ d_e)
       \geq\int_0^N\rho\bigl(V_\alpha^{\dagger}(s)\bigr)\,ds.
\end{equation}
\end{corollary}

\begin{proof}
The generalized inverse of \(r\mapsto\nualpha(\{d_e\leq r\})\)
agrees with \(V_\alpha^{\dagger}\).  Indeed, the closed ball of
radius \(r\) lies in every open ball of radius \(r+\eps\).
Continuity and strict monotonicity of \(\rho\) then give
\[
    F_{\alpha,\rho\circ d_e}^{\dagger}(s)
       =\rho\bigl(V_\alpha^{\dagger}(s)\bigr).
\]
Apply Theorem~\ref{thm:occupation-filling}; the value of the inverse
at a terminal endpoint does not affect its integral.
\end{proof}

\subsection{Matching bounds from kernel moments}
\label{subsec:kernel-moment-spectral-bounds}

The lower filling bound uses only exclusion.  To bound the spectral
cost from above, we quantify the extra energy introduced by averaging
against the reproducing kernel.

For the upper estimates in this and the next subsection, we use the
compact notation
\[
    \mathcal E_{\alpha,a}(\gamma)
       :=\int_X a(x)\langle\gamma k_x^\alpha,k_x^\alpha\rangle_\alpha
                    \,d\nu_\alpha(x),
    \qquad
    S_{\alpha,a}(N)
       :=\sum_{j=1}^n\lambda_{j,\alpha}(a)
                    +\theta\lambda_{n+1,\alpha}(a),
\]
where \(N=n+\theta\), \(n\in\mathbb N\cup\{0\}\), and
\(0\leq\theta<1\).  Thus \(\mathcal E_{\alpha,a}(\gamma)\)
is the total Toeplitz energy of the state \(\gamma\), and
\(S_{\alpha,a}(N)\) is the interpolated sum of the \(N\) lowest
eigenvalues of \(T_a^\alpha\).  We also write
\(\Phi_{\alpha,a}(N)=\int_0^N F_{\alpha,a}^{\dagger}(s)\,ds\)
for the corresponding scalar filling cost.

\begin{proposition}
\label{prop:additive-filling-error}
If, for some finite \(b_\alpha\geq0\),
\[
    \widetilde a^{\,\alpha}(x)\leq a(x)+b_\alpha
       \quad\text{for \(\nualpha\)-almost every }x,
\]
then, for every finite \(0\leq N\leq D_\alpha\),
\[
    \Phi_{\alpha,a}(N)
       \leq S_{\alpha,a}(N)
       \leq \Phi_{\alpha,a}(N)+b_\alpha N.
\]
\end{proposition}

\begin{proof}
Let \(h_N\) minimize the scalar filling problem of
Lemma~\ref{lem:filling-duality}.  The Toeplitz operator
\(\gamma_N=T_{h_N}^\alpha\) satisfies
\(0\leq\gamma_N\leq I\) and \(\Tr\gamma_N=N\).
Its occupation function is \(\widetilde{h_N}^{\,\alpha}\).
Symmetry of the kernel and Tonelli's theorem give
\begin{equation}\label{eq:toeplitz-trial-energy}
    \mathcal E_{\alpha,a}(\gamma_N)
       =\int_Xa\,\widetilde{h_N}^{\,\alpha}\,d\nualpha
       =\int_Xh_N\widetilde a^{\,\alpha}\,d\nualpha.
\end{equation}
Consequently,
\[
    S_{\alpha,a}(N)\leq\mathcal E_{\alpha,a}(\gamma_N)
       \leq\Phi_{\alpha,a}(N)+b_\alpha N.
\]
The lower bound is Theorem~\ref{thm:occupation-filling}.
All identities are valid as nonnegative extended integrals.
\end{proof}

For the distance cost, let \(N=n+\theta\) be a real occupation number.
A sufficient condition for an upper bound is the uniform first kernel moment
\begin{equation}\label{eq:uniform-kernel-radius}
    \eps_\alpha
       :=\mathop{\operatorname{ess\,sup}}_{x\in X}
             \int_Xd(x,y)\,Q_\alpha(x,dy)<\infty,
\end{equation}
where the essential supremum is relative to \(\nualpha\).
The triangle inequality gives
\(\widetilde{d_e}^{\,\alpha}\leq d_e+\eps_\alpha\), and hence
\begin{equation}\label{eq:radial-two-sided-sum}
    \int_0^NV_\alpha^{\dagger}(s)\,ds
       \leq S_{\alpha,d_e}(N)
       \leq\int_0^NV_\alpha^{\dagger}(s)\,ds+\eps_\alpha N.
\end{equation}
If a transitive group of isometries preserves \(\nualpha\) and the
kernel \(|\langle\kalpha{x},\kalpha{y}\rangle_\alpha|^2\), the
moment in \eqref{eq:uniform-kernel-radius} is independent of \(x\).
It then equals \(\widetilde{d_e}^{\,\alpha}(e)\).
No such symmetry is needed for \eqref{eq:radial-two-sided-sum}.

These upper bounds are especially natural in homogeneous models.  Suppose a
group acts transitively on \(X\) by isometries, preserves \(\nu_\alpha\),
and carries \(k_x^\alpha\) to \(k_{gx}^\alpha\) up to a phase.  Then
\(Q_\alpha(gx,gE)=Q_\alpha(x,E)\), the volume of a ball is independent
of its center, and the kernel displacement has the same law at every point.
More generally, for \(p>0\),
\[
    \eps_{\alpha,p}^p
       =\int_X d(e,y)^p\,Q_\alpha(e,dy).
\]
The essential supremum in the moment hypothesis is therefore a single
centered moment rather than an additional global assumption.

For the radial cost \(d_e^p\), averaging by the kernel amounts to moving a
point a random distance with this fixed distribution.  The triangle
inequality, followed by Minkowski's inequality when \(p\geq1\), gives the
upper symbol bound used below.  Thus the scalar filling state acquires only a
uniform displacement cost under Toeplitz averaging.  In the Fock, weighted
Bergman, and Gabor models this covariance is built into the geometry.  If the
filled radius tends to infinity, a fixed displacement is lower order.  In the
standard homogeneous semiclassical scalings, where the overlap measures
concentrate at the identity with a controlled \(p\)-th moment, the displacement
moment itself tends to zero.

\begin{corollary}
\label{cor:power-cost-sandwich}
Let \(p>0\), and suppose
\[
    \eps_{\alpha,p}
       :=\mathop{\operatorname{ess\,sup}}_{x\in X}
          \left(\int_Xd(x,y)^p\,Q_\alpha(x,dy)\right)^{1/p}<\infty.
\]
Write \(N=n+\theta\), where \(n\in\mathbb N\cup\{0\}\) and
\(0\leq\theta<1\).  For \(p\geq1\),
\begin{equation}\label{eq:power-cost-sandwich}
    \int_0^N\bigl(V_\alpha^{\dagger}(s)\bigr)^p\,ds
       \leq S_{\alpha,d_e^p}(N)
       \leq\int_0^N
                    \bigl(V_\alpha^{\dagger}(s)+\eps_{\alpha,p}\bigr)^p\,ds.
\end{equation}
For \(0<p<1\), the corresponding estimate is
\begin{equation}\label{eq:fractional-cost-sandwich}
    \int_0^N\bigl(V_\alpha^{\dagger}(s)\bigr)^p\,ds
       \leq S_{\alpha,d_e^p}(N)
       \leq\int_0^N\bigl(V_\alpha^{\dagger}(s)\bigr)^p\,ds
                    +\eps_{\alpha,p}^pN.
\end{equation}
If \(\Halpha\) is infinite dimensional, then, at this fixed parameter,
\begin{equation}\label{eq:fixed-frame-filling-asymptotic}
    S_{\alpha,d_e^p}(N)
       \sim\int_0^N\bigl(V_\alpha^{\dagger}(s)\bigr)^p\,ds
       \qquad(N\to\infty).
\end{equation}
\end{corollary}

\begin{proof}
For \(p\geq1\), Minkowski's inequality gives
\[
    \bigl(\widetilde{d_e^p}^{\,\alpha}(x)\bigr)^{1/p}
       \leq d_e(x)+\eps_{\alpha,p}
\]
for almost every \(x\).  Choose the scalar filling density for
\(d_e^p\).  It fills the smallest radii first, so
\eqref{eq:toeplitz-trial-energy} bounds its Toeplitz trial energy by
the last integral in \eqref{eq:power-cost-sandwich}.
The lower bound is Corollary~\ref{cor:radial-filling}.
For \(0<p<1\), use
\(d_e(y)^p\leq d_e(x)^p+d(x,y)^p\), followed by
Proposition~\ref{prop:additive-filling-error}.

In infinite dimension, \(\nualpha(X)=\infty\), whereas every ball
has finite measure.  Thus \(V_\alpha^{\dagger}(s)\to\infty\).
The ratio
\((V_\alpha^{\dagger}(s)+\eps_{\alpha,p})^p/
 (V_\alpha^{\dagger}(s))^p\) tends to one.  Integrating after any
fixed lower cutoff, and observing that the contribution before that
cutoff is negligible, proves
\eqref{eq:fixed-frame-filling-asymptotic} for \(p\geq1\).
For \(0<p<1\), the same conclusion follows from
\eqref{eq:fractional-cost-sandwich}, since
\(N^{-1}\int_0^N(V_\alpha^{\dagger}(s))^p\,ds\to\infty\).
\end{proof}

The uniform moment assumptions above serve a different purpose from
\eqref{eq:kernel-moment-tail}: they compare the entire spectrum with
volume at one fixed parameter.  They need not hold in every model.
The lower occupation bounds remain available without them.
Fractional powers are useful when the first kernel moment is infinite
but a smaller positive moment is finite.

\subsection{Polynomial and exponential volume growth}

\begin{corollary}\label{cor:volume-growth-bounds}
Suppose that, for all \(r>0\),
\[
    V_\alpha(r)\leq A_\alpha r^D,
    \qquad A_\alpha>0,\quad D>0.
\]
For every \(p>0\) for which the radial form is densely defined, and
\(N=n+\theta\) with \(n\in\mathbb N\cup\{0\}\) and \(0\leq\theta<1\),
\begin{equation}\label{eq:polynomial-spectral-bound}
    S_{\alpha,d_e^p}(N)
       \geq\frac{D}{D+p}\,A_\alpha^{-p/D}N^{1+p/D}.
\end{equation}
If instead
\(V_\alpha(r)\leq A_\alpha(e^{\beta r}-1)\), with \(\beta>0\),
then
\[
    S_{\alpha,d_e}(N)
       \geq\frac1\beta
          \left[(N+A_\alpha)\log\left(1+\frac{N}{A_\alpha}\right)-N\right].
\]
The same lower bounds hold with the eigenvalue sums replaced by
\(\mathcal E_{\alpha,d_e^p}(\gamma)\) and
\(\mathcal E_{\alpha,d_e}(\gamma)\), respectively, whenever
\(0\leq\gamma\leq I\) and \(\Tr\gamma=N\).
\end{corollary}

\begin{proof}
The polynomial bound implies
\(V_\alpha^{\dagger}(s)\geq(s/A_\alpha)^{1/D}\).
Integrate its \(p\)-th power in
\eqref{eq:radial-spectral-filling}.  In the exponential case use
\(V_\alpha^{\dagger}(s)\geq\beta^{-1}\log(1+s/A_\alpha)\)
and integrate.  The density-matrix statements follow from
\eqref{eq:occupation-filling-chain}.
\end{proof}

In the common scaling \(A_\alpha=Cc(\alpha)\), writing
\(N=n+\theta\) as above, \eqref{eq:polynomial-spectral-bound} gives
\[
    \frac{S_{\alpha,d_e^p}(N)}{N}
       \geq\frac{D}{D+p}\,C^{-p/D}
                \left(\frac{N}{c(\alpha)}\right)^{p/D}.
\]
Thus the relevant density is \(N/c(\alpha)\).
For exponential growth, the lower bound on the average distance is
\(\beta^{-1}\log(N/A_\alpha)+O(1)\) as \(N/A_\alpha\to\infty\).
When the corresponding uniform kernel moment is finite, these rates
are sharp whenever the volume has the matching asymptotic behavior.
For example, at a fixed parameter, if
\(V_\alpha(r)\sim A_\alpha r^D\), then
Corollary~\ref{cor:power-cost-sandwich} gives
\[
    S_{\alpha,d_e^p}(N)
       \sim\frac{D}{D+p}\,A_\alpha^{-p/D}N^{1+p/D},\qquad p>0.
\]
If \(V_\alpha(r)\sim A_\alpha e^{\beta r}\) and
\(\eps_\alpha<\infty\), then
\(
  S_{\alpha,d_e}(N)
  \sim\beta^{-1}N\log N
\).

\subsection{Semiclassical spectral filling}
\label{subsec:semiclassical-filling}

We now let \(\alpha\to\infty\).  For a continuous nonnegative
confining symbol \(a\), write
\[
\begin{gathered}
    F_a(t)=\nu(\{a\leq t\}),\qquad
    F_a^{\dagger}(s)=\inf\{t\geq0:F_a(t)\geq s\},\\
    \Phi_a(\varrho)=\int_0^{\varrho}F_a^{\dagger}(s)\,ds.
\end{gathered}
\]
The symbol \(\varrho\) denotes the limiting number of occupied states
per unit of \(c(\alpha)\); it is distinct from the radial profile
\(\rho\).

\begin{theorem}
\label{thm:semiclassical-filling}
Suppose that the forms \(q_{\alpha,a}\) are densely defined and
\(m_\alpha\rightharpoonup\nu\) as in
\eqref{eq:measure-convergence}.  Let \(0\leq N_\alpha\leq D_\alpha\)
be finite, write \(N_\alpha=n_\alpha+\theta_\alpha\) with
\(n_\alpha\in\mathbb N\cup\{0\}\) and \(0\leq\theta_\alpha<1\), and
assume \(N_\alpha/c(\alpha)\to\varrho\in[0,\infty)\).
Then
\[
    \liminf_{\alpha\to\infty}
       \frac{1}{c(\alpha)}
       \left(\sum_{j=1}^{n_\alpha}\lambda_{j,\alpha}(a)
              +\theta_\alpha\lambda_{n_\alpha+1,\alpha}(a)\right)
       \geq\int_0^\varrho F_a^{\dagger}(s)\,ds.
\]
This lower bound needs neither averaged localization nor a kernel
moment hypothesis.

If, in addition, \eqref{eq:averaged-localization} and
\eqref{eq:kernel-moment-tail} hold and
\(0<\varrho<\nu(X)\), then
\begin{equation}\label{eq:semiclassical-filling-limit}
    \frac{1}{c(\alpha)}
       \left(\sum_{j=1}^{n_\alpha}\lambda_{j,\alpha}(a)
              +\theta_\alpha\lambda_{n_\alpha+1,\alpha}(a)\right)
       \longrightarrow\int_0^\varrho F_a^{\dagger}(s)\,ds.
\end{equation}
The total limiting volume \(\nu(X)\) may be finite or infinite,
and the distribution of \(a\) may have atoms.
\end{theorem}

\begin{proof}
Put \(\varrho_\alpha=N_\alpha/c(\alpha)\) and
\[
    G_\alpha(t)=\int_X(t-a)_+\,dm_\alpha,
    \qquad G(t)=\int_X(t-a)_+\,d\nu.
\]
For each fixed \(t\geq0\), the integrand is continuous and compactly
supported, so \(G_\alpha(t)\to G(t)\).  By the Riesz upper bound
and \eqref{eq:spectral-filling-duality},
\[
    \frac{S_{\alpha,a}(N_\alpha)}{c(\alpha)}
       \geq\varrho_\alpha t-G_\alpha(t).
\]
Take the lower limit and then the supremum over \(t\geq0\).
Lemma~\ref{lem:filling-duality}, applied to \(\nu\), identifies the
result as \(\Phi_a(\varrho)\).  If \(\varrho>\nu(X)\), both
this supremum and \(\Phi_a(\varrho)\) are infinite; the same
argument still applies.

For the matching limit set
\(r_\alpha(t)=R_{\alpha,a}(t)/c(\alpha)\).
Theorem~\ref{thm:riesz-limit} gives \(r_\alpha(t)\to G(t)\)
for every \(t\geq0\).  These are nondecreasing functions of \(t\)
with continuous limit, so the convergence is uniform on every compact
interval.  Since \(F_a(t)\uparrow\nu(X)>\varrho\), choose
\(L>0\) such that
\[
    G(L+1)-G(L)=\int_L^{L+1}F_a(u)\,du>\varrho.
\]
For all sufficiently large \(\alpha\),
\(r_\alpha(L+1)-r_\alpha(L)>\varrho_\alpha\).
Convexity implies, for \(t\geq L+1\),
\[
    r_\alpha(t)\geq r_\alpha(L+1)
       +(t-L-1)\bigl(r_\alpha(L+1)-r_\alpha(L)\bigr).
\]
Consequently the supremum of
\(\varrho_\alpha t-r_\alpha(t)\) is attained in \([0,L+1]\).
The same is true of \(\varrho t-G(t)\).  Uniform convergence on
this interval and \eqref{eq:spectral-filling-duality} now give
\[
    \frac{S_{\alpha,a}(N_\alpha)}{c(\alpha)}
       \longrightarrow\sup_{t\geq0}\{\varrho t-G(t)\}
       =\Phi_a(\varrho).
\]
\end{proof}

For the distance symbol this becomes
\[
    \frac1{c(\alpha)}\sum_{j=1}^{N_\alpha}\lambda_{j,\alpha}(d_e)
       \longrightarrow\int_0^{\varrho}V^{\dagger}(s)\,ds,
    \qquad V(r)=\nu(B(e,r)),
\]
when \(N_\alpha\) is integer and the full hypotheses of the theorem
hold.  In contrast to a counting limit at a prescribed energy, this
integrated statement requires no condition on the mass of a sphere.

\begin{corollary}
\label{cor:eigenvalue-quantile-limit}
Assume the hypotheses of
\eqref{eq:semiclassical-filling-limit}, with integer \(N_\alpha\).
Suppose that some \(t_\varrho\geq0\) satisfies
\begin{equation}\label{eq:strict-quantile-crossing}
    F_a(t_\varrho-\delta)<\varrho<F_a(t_\varrho+\delta)
       \qquad\text{for every }\delta>0,
\end{equation}
where \(F_a(t)=0\) for \(t<0\).  Then
\[
    \lambda_{N_\alpha,\alpha}(a)\longrightarrow t_\varrho.
\]
\end{corollary}

\begin{proof}
Choose positive continuity points \(t_-<t_\varrho<t_+\) of
\(F_a\), arbitrarily close to \(t_\varrho\).
Condition~\eqref{eq:strict-quantile-crossing} and
Corollary~\ref{cor:counting-limit} imply, for large \(\alpha\),
\[
    \Tr\mathbf1_{[0,t_-]}(T_a^\alpha)<N_\alpha
       <\Tr\mathbf1_{[0,t_+]}(T_a^\alpha).
\]
Thus \(t_-<\lambda_{N_\alpha,\alpha}(a)\leq t_+\).
When \(t_\varrho=0\), use nonnegativity in place of the lower
counting bound.  Let the comparison points approach \(t_\varrho\).
\end{proof}

At full limiting occupation, a separate endpoint argument is needed.
For compact \(X\), it is immediate: if
\(N_\alpha/c(\alpha)\to\nu(X)\), then
\(D_\alpha/c(\alpha)=m_\alpha(X)\to\nu(X)\), and
\[
\begin{split}
    0&\leq\Tr T_a^\alpha-S_{\alpha,a}(N_\alpha)
       \leq\|a\|_\infty(D_\alpha-N_\alpha),\\
    \frac1{c(\alpha)}\Tr T_a^\alpha
       &=\int_Xa\,dm_\alpha\longrightarrow\int_Xa\,d\nu.
\end{split}
\]
Hence \eqref{eq:semiclassical-filling-limit} also holds at
\(\varrho=\nu(X)\) in that setting.  On a noncompact space,
weak convergence alone does not control the energy carried by mass
escaping to infinity.

We have obtained both the exact price of the scalar filling problem
and conditions under which the Toeplitz spectrum attains that price
asymptotically.  The uncertainty principles in the next section follow
by bounding the energy of a family from above using its centers,
dispersions, or a common envelope, and comparing with these lower bounds.

% -----------------------------------------------------------------------------
\section{Collective uncertainty principles}
\label{sec:collective-uncertainty}

A localization hypothesis limits where a state can spend its mass.
Orthogonality limits how much mass the whole family can spend there.
Comparing these two constraints gives the uncertainty principles in this
section.  Moments lead to mean-dispersion estimates; common envelopes
lead to cardinality bounds; and the averaged phase-space measure records
the resulting obstruction in Wasserstein distance.

We work on a proper metric space and use the notation of
Section~\ref{sec:radial-spectra}.  The inequalities at a fixed parameter
use only the normalized Parseval frame and its actual measure
\(\nualpha\).  They do not require measure convergence, kernel
localization, or a densely defined unbounded Toeplitz operator.
When the relevant form is densely defined, the spectral sums of
Section~\ref{sec:radial-spectra} give the stated refinements.

\subsection{Centers and dispersions}

For \(1\leq p<\infty\), let \(\mathcal P_p(X)\) denote the probability
measures with finite \(p\)-th moment.  For \(\eta\in\mathcal P_p(X)\), put
\[
    M_p(\eta;y)=\left(\int_Xd(x,y)^p\,d\eta(x)\right)^{1/p},
    \qquad \Delta_p(\eta)=\inf_{y\in X}M_p(\eta;y).
\]
Recall that a unit vector \(f\in\Halpha\) determines the probability
measure
\[
    d\mu_f^\alpha(x)
       =|\langle f,k_x^\alpha\rangle_\alpha|^2\,d\nu_\alpha(x).
\]
If \(\mu_f^\alpha\in\mathcal P_p(X)\), choose
\[
    \mu_{\alpha,p}(f)\in\argmin_{y\in X}M_p(\mu_f^\alpha;y),
    \qquad
    \Delta_{\alpha,p}(f)=\Delta_p(\mu_f^\alpha).
\]
For \(p=1\), these are the median and dispersion
\(\mu_\alpha(f)\) and \(\Delta_\alpha(f)\) used in the introduction.

\begin{lemma}
\label{lem:moment-centers}
Every \(\eta\in\mathcal P_p(X)\) has a minimizing center.  Moreover,
for any \(y,z\in X\),
\[
    |M_p(\eta;y)-M_p(\eta;z)|\leq d(y,z).
\]
In particular, for every unit vector with finite \(p\)-th moment,
\[
    \left(q_{\alpha,d_e^p}(f)\right)^{1/p}
       \leq\Delta_{\alpha,p}(f)+d(e,\mu_{\alpha,p}(f)).
\]
\end{lemma}

\begin{proof}
The triangle inequality in \(X\), followed by Minkowski's inequality,
gives \(M_p(\eta;y)\leq M_p(\eta;z)+d(y,z)\).
Also,
\[
    M_p(\eta;y)\geq d(e,y)-M_p(\eta;e).
\]
Thus the continuous function \(y\mapsto M_p(\eta;y)\) has bounded
sublevel sets.  Properness makes those sets compact, so a minimum
exists.  Apply the first inequality with
\(y=e\), \(z=\mu_{\alpha,p}(f)\), and \(\eta=\mu_f^\alpha\).
\end{proof}

\subsection{Mean-dispersion inequalities}

\begin{theorem}
\label{thm:mean-dispersion}
Let \(\varphi_1,\ldots,\varphi_N\) be orthonormal in \(\Halpha\),
with \(\mu_{\varphi_j}^\alpha\in\mathcal P_p(X)\).  Then
\begin{equation}\label{eq:p-mean-dispersion}
\begin{aligned}
    \sum_{j=1}^N
       \bigl[\Delta_{\alpha,p}(\varphi_j)
                  +d(e,\mu_{\alpha,p}(\varphi_j))\bigr]^p
       &\geq\sum_{j=1}^Nq_{\alpha,d_e^p}(\varphi_j)\\
       &\geq\int_0^N\bigl(V_\alpha^\dagger(s)\bigr)^p\,ds.
\end{aligned}
\end{equation}
Consequently,
\begin{equation}\label{eq:separate-mean-dispersion}
    \sum_{j=1}^N
       \left[\Delta_{\alpha,p}(\varphi_j)^p
                  +d(e,\mu_{\alpha,p}(\varphi_j))^p\right]
       \geq 2^{1-p}\int_0^N\bigl(V_\alpha^\dagger(s)\bigr)^p\,ds.
\end{equation}
If \(q_{\alpha,d_e^p}\) is densely defined, one may replace the
integral on the right of \eqref{eq:p-mean-dispersion} by
\(\sum_{j=1}^N\lambda_{j,\alpha}(d_e^p)\).
\end{theorem}

\begin{proof}
The first inequality follows from
Lemma~\ref{lem:moment-centers}.  The occupation function
\[
    s_\alpha(x)=\sum_{j=1}^N
                     |\langle\varphi_j,\kalpha{x}\rangle_\alpha|^2
\]
satisfies \(0\leq s_\alpha\leq1\) and
\(\int_Xs_\alpha\,d\nualpha=N\).
Apply Lemma~\ref{lem:filling-duality} to \(a=d_e^p\).
The scalar filling cost is the integral in
\eqref{eq:p-mean-dispersion}, by the change of radial variable in
Corollary~\ref{cor:radial-filling}.  Finally,
\((u+v)^p\leq2^{p-1}(u^p+v^p)\) gives
\eqref{eq:separate-mean-dispersion}.  The spectral refinement is
Theorem~\ref{thm:occupation-filling}.
\end{proof}

At \(p=1\), \eqref{eq:p-mean-dispersion} is the form stated in the
introduction.  The center term records the
possibility of keeping individual states localized while moving them
farther apart.  The collective form of the estimate connects it with
the mean-dispersion principles of Jaming--Powell
\cite{JamingPowell} and Malinnikova \cite{Malinnikova}; here the moments
are those of the frame-induced phase-space measures.

For Euclidean \(X\) and \(p=2\), the minimizing center is the
barycenter and the identity
\[
    q_{\alpha,d_e^2}(f)
       =\Delta_{\alpha,2}(f)^2+|\mu_{\alpha,2}(f)-e|^2
\]
improves the factor \(2^{1-p}\) in
\eqref{eq:separate-mean-dispersion} to \(1\).  This follows by
expanding the square and using the vanishing first moment about the
barycenter.  In the Gaussian Gabor model below, this quadratic identity
converts the exact phase-space filling spectrum into an explicit
finite-\(N\) version of Malinnikova's time--frequency moment inequality.

\begin{corollary}
\label{cor:bounded-localization-cardinality}
Suppose \(V_\alpha(r)\leq A_\alpha r^D\) for every \(r>0\), with
\(A_\alpha,D>0\).  If
\[
    \Delta_{\alpha,p}(\varphi_j)\leq L,\qquad
    d(e,\mu_{\alpha,p}(\varphi_j))\leq R
       \quad(1\leq j\leq N),
\]
then
\begin{equation}\label{eq:bounded-localization-cardinality}
    N\leq A_\alpha
            \left(\frac{D+p}{D}\right)^{D/p}(L+R)^D.
\end{equation}
At any fixed parameter, properness alone rules out an infinite
orthonormal sequence whose centers stay in a bounded set and whose
\(p\)-dispersions stay bounded.
\end{corollary}

\begin{proof}
The volume bound gives
\[
    N(L+R)^p
       \geq\int_0^N(V_\alpha^\dagger(s))^p\,ds
       \geq\frac{D}{D+p}A_\alpha^{-p/D}N^{1+p/D}.
\]
Rearrange to obtain \eqref{eq:bounded-localization-cardinality}.
For the last assertion, an infinite orthonormal sequence would imply
\(\nualpha(X)=\infty\).  Every fixed ball has finite measure, so
\(V_\alpha^\dagger(s)\to\infty\) and
\[
    \frac1N\int_0^N(V_\alpha^\dagger(s))^p\,ds\longrightarrow\infty,
\]
contradicting the uniform upper bound \((L+R)^p\).
\end{proof}

\begin{corollary}
\label{cor:semiclassical-mean-dispersion}
Assume only \eqref{eq:measure-convergence}, and let the orthonormal
families have sizes \(N_\alpha\) with
\(N_\alpha/c(\alpha)\to\varrho\in(0,\infty)\).
Fix \(p\geq1\), and assume that all their phase-space measures
belong to \(\mathcal P_p(X)\).  Then
\begin{equation}\label{eq:semiclassical-mean-dispersion}
\begin{split}
    &\liminf_{\alpha\to\infty}\frac1{c(\alpha)}
       \sum_{j=1}^{N_\alpha}
          \bigl[\Delta_{\alpha,p}(\varphi_{j,\alpha})
               +d(e,\mu_{\alpha,p}(\varphi_{j,\alpha}))\bigr]^p\\
    &\hspace{35mm}\geq\int_0^\varrho(V^\dagger(s))^p\,ds,
       \qquad V(r)=\nu(B(e,r)).
\end{split}
\end{equation}
No convergence of the moments of individual states is assumed.
\end{corollary}

\begin{proof}
For every \(t\geq0\), occupation and the scalar filling duality give
\[
    \frac1{c(\alpha)}
       \sum_{j=1}^{N_\alpha}q_{\alpha,d_e^p}(\varphi_{j,\alpha})
       \geq\frac{N_\alpha}{c(\alpha)}t
                  -\int_X(t-d_e^p)_+\,dm_\alpha.
\]
The last integrand is continuous and compactly supported.  Pass to
the lower limit, take the supremum over \(t\), and apply
Lemma~\ref{lem:filling-duality} to \(\nu\).
If \(\varrho>\nu(X)\), both the supremum and the claimed lower bound
are infinite.
Lemma~\ref{lem:moment-centers} bounds the left side from above by
the expression in \eqref{eq:semiclassical-mean-dispersion}.
\end{proof}

\subsection{Umbrellas, effective volume, and rearrangement}
\label{subsec:umbrella-rearrangement}

An umbrella for a family is a nonnegative Borel function
\(\Phi_\alpha\) such that
\begin{equation}\label{eq:common-frame-envelope}
    |\langle\varphi_j,\kalpha{x}\rangle_\alpha|^2
       \leq\Phi_\alpha(x)
       \quad\text{for \(\nualpha\)-almost every }x,\quad1\leq j\leq N.
\end{equation}
Its total mass may exceed one; no normalization is imposed.
All exceptional sets and all integrals here refer to \(\nualpha\).

For \(\Phi_\alpha\in L^1(\nualpha)\) and \(0<\delta<1\), define
\begin{equation}\label{eq:effective-envelope-volume}
    K_{\alpha,\Phi_\alpha}(\delta)
       =\inf\left\{\nualpha(E):
          \int_{X\setminus E}\Phi_\alpha\,d\nualpha\leq\delta\right\},
\end{equation}
where \(E\) ranges over Borel sets.  This is the least frame volume
capturing all but \(\delta\) of the envelope mass.
Write \(\Phi_\alpha^*\) for the nonincreasing rearrangement relative
to \(\nualpha\), on the interval \((0,\nualpha(X))\), extended by
zero beyond that interval when the total measure is finite.  Thus
\(\Phi_\alpha\) and \(\Phi_\alpha^*\) have the same superlevel
measures.  This definition also applies to atomic frame measures.

\begin{theorem}
\label{thm:rearrangement-umbrella}
Let an orthonormal family of size \(N\geq1\) satisfy
\eqref{eq:common-frame-envelope}, with
\(\Phi_\alpha\in L^1(\nualpha)\).  Then
\begin{equation}\label{eq:envelope-capacity}
    N\leq\int_X\min\{1,N\Phi_\alpha(x)\}\,d\nualpha(x)
       =\int_0^\infty\min\{1,N\Phi_\alpha^*(s)\}\,ds.
\end{equation}
For every \(0<\delta<1\),
\[
    N(1-\delta)\leq K_{\alpha,\Phi_\alpha}(\delta).
\]
In particular, an infinite orthonormal sequence cannot have a common
integrable frame envelope.
\end{theorem}

\begin{proof}
The occupation function satisfies both \(s_\alpha\leq1\) and
\(s_\alpha\leq N\Phi_\alpha\).  Integrating proves the inequality
in \eqref{eq:envelope-capacity}.  Its equality follows from
equimeasurability, or directly from layer cake:
\[
    \int_X\min\{1,N\Phi_\alpha\}\,d\nualpha
       =\int_0^1\nualpha(\{\Phi_\alpha>t/N\})\,dt.
\]
If \(E\) is admissible in \eqref{eq:effective-envelope-volume}, each
state has at least \(1-\delta\) of its mass in \(E\).  Therefore
\[
    N(1-\delta)
       \leq\int_Es_\alpha\,d\nualpha\leq\nualpha(E).
\]
Take the infimum over \(E\).
Finally, dividing \eqref{eq:envelope-capacity} by \(N\) gives
\[
    1\leq\int_X\min\{N^{-1},\Phi_\alpha\}\,d\nualpha.
\]
At a fixed parameter the right side tends to zero as \(N\to\infty\),
by dominated convergence.  This excludes an infinite family.
\end{proof}

The capacity condition retains all the information furnished by the
pointwise bounds \(s_\alpha\leq1\) and
\(s_\alpha\leq N\Phi_\alpha\).  Indeed,
\[
    \inf_E\left\{\nualpha(E)
          +N\int_{X\setminus E}\Phi_\alpha\,d\nualpha\right\}
       =\int_X\min\{1,N\Phi_\alpha\}\,d\nualpha,
\]
with a minimizing choice \(E=\{N\Phi_\alpha>1\}\).
Thus \eqref{eq:envelope-capacity} is the result of optimizing all
such cutoff estimates.  These are necessary conditions for an
orthonormal family; constructing vectors with a prescribed envelope
also depends on the RKHS.

The qualitative conclusion is a frame version of Shapiro's umbrella
principle \cite{Shapiro}; quantitative predecessors include
Jaming--Powell \cite{JamingPowell}, and general frame compactness
methods appear in \cite{MSWWCompactness}.  For umbrella theorems
formulated through Wigner, Rihaczek, and Mellin representations, see
Boggiatto--Boiti--Oliaro \cite{BoggiattoBoitiOliaro}.

\begin{proposition}
\label{prop:moment-umbrella}
Suppose \eqref{eq:common-frame-envelope} holds.  For any nonnegative
continuous confining cost \(a\), if
\[
    M_{\alpha,a}(\Phi_\alpha)
       :=\int_Xa(x)\Phi_\alpha(x)\,d\nualpha(x)<\infty,
\]
then
\begin{equation}\label{eq:moment-umbrella}
    \int_0^N F_{\alpha,a}^{\dagger}(s)\,ds
       \leq N M_{\alpha,a}(\Phi_\alpha).
\end{equation}
If the form is densely defined and \(N=n+\theta\), the left side may be
replaced by
\(\sum_{j=1}^n\lambda_{j,\alpha}(a)
 +\theta\lambda_{n+1,\alpha}(a)\).
For \(a=d_e^p\), \(p>0\), and
\(V_\alpha(r)\leq A_\alpha r^D\), this implies
\begin{equation}\label{eq:polynomial-umbrella}
    N\leq A_\alpha
       \left(\frac{D+p}{D}\right)^{D/p}
           \bigl(M_{\alpha,d_e^p}(\Phi_\alpha)\bigr)^{D/p}.
\end{equation}
\end{proposition}

\begin{proof}
The envelope gives
\(\sum_jq_{\alpha,a}(\varphi_j)
 \leq N M_{\alpha,a}(\Phi_\alpha)\).
The occupation argument of Lemma~\ref{lem:filling-duality} bounds
this sum from below by \(\Phi_{\alpha,a}(N)\).
Use Theorem~\ref{thm:occupation-filling} for the spectral refinement.
In the radial polynomial case, integrate
\((V_\alpha^\dagger(s))^p\geq(s/A_\alpha)^{p/D}\) and rearrange.
\end{proof}

To see the relation with effective volume, write
\(M=M_{\alpha,d_e^p}(\Phi_\alpha)\).
If \(M>0\), then, for \(r>0\),
\[
    \int_{X\setminus B(e,r)}\Phi_\alpha\,d\nualpha\leq M r^{-p}.
\]
Hence, whenever \(r^p>M\), the same cutoff argument gives
\[
    N\leq\frac{V_\alpha(r)}{1-M/r^p}.
\]
In particular, \(N\leq2V_\alpha((2M)^{1/p})\).
The integrated bound \eqref{eq:moment-umbrella} uses the entire
volume profile and yields the better polynomial constant in
\eqref{eq:polynomial-umbrella}.  If \(M=0\), all states are supported,
in frame coefficients, at \(e\), so \(N\leq\nualpha(\{e\})\leq1\).
These moment estimates are valid also when the envelope itself is
not integrable near the zero set of the cost.

\subsection{Transport away from a prescribed distribution}
\label{subsec:collective-transport}

For \(\eta,\sigma\in\mathcal P_p(X)\), the Wasserstein distance is
\[
    W_p(\eta,\sigma)
       =\left(\inf_{\pi\in\Pi(\eta,\sigma)}
                    \int_{X\times X}d(x,y)^p\,d\pi(x,y)\right)^{1/p}.
\]
The only coupling of \(\eta\) with \(\delta_y\) is
\(\eta\otimes\delta_y\), so
\[
    W_p(\eta,\delta_y)=M_p(\eta;y),
    \qquad
    \Delta_p(\eta)=\min_{y\in X}W_p(\eta,\delta_y).
\]
In particular,
\(q_{\alpha,d_e^p}(f)=W_p(\mu_f^\alpha,\delta_e)^p\).

Let \(0\leq\gamma\leq I\) be trace class with \(N=\Tr\gamma>0\), and
define its normalized phase-space measure by
\[
    d\overline\mu_\gamma^\alpha(x)
       =\frac{h_\gamma^\alpha(x)}{N}\,d\nualpha(x).
\]
For the projection onto an orthonormal family, this is
\(N^{-1}\sum_{j=1}^N\mu_{\varphi_j}^\alpha\).
To account for a possibly nonhomogeneous frame measure, write
\[
\begin{split}
    V_{\alpha,y}(r)&=\nualpha(B(y,r)),\\
    \mathcal R_{\alpha,p}(N;y)
       &=\left(\frac1N\int_0^N
                       (V_{\alpha,y}^\dagger(s))^p\,ds\right)^{1/p}.
\end{split}
\]

\begin{theorem}
\label{thm:collective-transport}
For every \(\sigma\in\mathcal P_p(X)\),
\begin{equation}\label{eq:center-dependent-transport}
    W_p(\overline\mu_\gamma^\alpha,\sigma)
       \geq\sup_{y\in X}
             \bigl[\mathcal R_{\alpha,p}(N;y)-M_p(\sigma;y)\bigr]_+.
\end{equation}
If a finite nondecreasing function
\(U_\alpha:[0,\infty)\to[0,\infty)\) satisfies
\begin{equation}\label{eq:uniform-ball-volume}
    \nualpha(B(y,r))\leq U_\alpha(r)
       \qquad(y\in X,\ r>0),
\end{equation}
then
\begin{equation}\label{eq:uniform-volume-transport}
    W_p(\overline\mu_\gamma^\alpha,\sigma)
       \geq
       \left[
          \left(\frac1N\int_0^N(U_\alpha^\dagger(s))^p\,ds\right)^{1/p}
          -\Delta_p(\sigma)
       \right]_+.
\end{equation}
Distances are interpreted as \(+\infty\) if
\(\overline\mu_\gamma^\alpha\) has infinite \(p\)-th moment.
\end{theorem}

\begin{proof}
The density of \(\overline\mu_\gamma^\alpha\) relative to
\(\nualpha\) is at most \(1/N\).  Scalar filling at any fixed
center \(y\) therefore gives
\[
    W_p(\overline\mu_\gamma^\alpha,\delta_y)^p
       =\frac1N\int_Xd(x,y)^p h_\gamma^\alpha(x)\,d\nualpha(x)
       \geq\mathcal R_{\alpha,p}(N;y)^p.
\]
If this measure has finite moment, the triangle inequality yields
\[
    W_p(\overline\mu_\gamma^\alpha,\sigma)+M_p(\sigma;y)
       \geq\mathcal R_{\alpha,p}(N;y).
\]
Take the supremum over \(y\) to prove
\eqref{eq:center-dependent-transport}.
Under \eqref{eq:uniform-ball-volume}, one has
\(V_{\alpha,y}^\dagger\geq U_\alpha^\dagger\).  The lower filling
radius is then independent of \(y\), and minimizing
\(M_p(\sigma;y)\) proves \eqref{eq:uniform-volume-transport}.
If the source has infinite moment, its distance to the finite-moment
target is infinite by the same moment triangle inequality.
\end{proof}

Thus a target's dispersion can replace its moment about a prescribed
center whenever a uniform volume bound is available.
This requires no transitive action or covariance of the kernels.
In particular, if
\(\nualpha(B(y,r))\leq A_\alpha r^D\) for every \(y,r\), then
\[
    W_p(\overline\mu_\gamma^\alpha,\sigma)
       \geq
       \left[
          \left(\frac{D}{D+p}\right)^{1/p}
          A_\alpha^{-1/D}N^{1/D}-\Delta_p(\sigma)
       \right]_+.
\]
For \(A_\alpha=Cc(\alpha)\), the leading term is a constant times
\((N/c(\alpha))^{1/D}\).
The same argument works for exponential volume by substituting its
generalized inverse in \eqref{eq:uniform-volume-transport}.

\begin{corollary}
\label{cor:transport-escape}
Fix \(\alpha\), and let \((\varphi_j)_{j\geq1}\) be orthonormal.
Set
\(\overline\mu_N^\alpha=N^{-1}\sum_{j=1}^N\mu_{\varphi_j}^\alpha\).
Then, for every fixed \(\sigma\in\mathcal P_p(X)\),
\[
    W_p(\overline\mu_N^\alpha,\sigma)\longrightarrow\infty
       \qquad(N\to\infty).
\]
\end{corollary}

\begin{proof}
Use \eqref{eq:center-dependent-transport} at the single fixed
center \(e\).  Since every ball has finite \(\nualpha\)-measure
and \(\nualpha(X)=\infty\), its filling radius
\(\mathcal R_{\alpha,p}(N;e)\) tends to infinity.
The target moment \(M_p(\sigma;e)\) is fixed and finite.
\end{proof}

The transport formulation also retains the distinction between
dispersion and motion of centers.  With
\[
    \sigma_N=\frac1N\sum_{j=1}^N
                         \delta_{\mu_{\alpha,p}(\varphi_j)},
\]
the average of the couplings from each state to its own center gives
\[
    W_p(\overline\mu_N^\alpha,\sigma_N)^p
       \leq\frac1N\sum_{j=1}^N\Delta_{\alpha,p}(\varphi_j)^p.
\]
Combining this with the triangle inequality to \(\delta_e\) yields
\[
\begin{split}
    \mathcal R_{\alpha,p}(N;e)
       &\leq
       \left(\frac1N\sum_{j=1}^N
                       \Delta_{\alpha,p}(\varphi_j)^p\right)^{1/p}\\
       &\quad+
       \left(\frac1N\sum_{j=1}^N
                   d(e,\mu_{\alpha,p}(\varphi_j))^p\right)^{1/p}.
\end{split}
\]
The family must pay through its average dispersion, the spread of its
centers, or both.

\begin{corollary}
\label{cor:semiclassical-transport}
Assume \eqref{eq:measure-convergence}, and let
\(0\leq\gamma_\alpha\leq I\) have finite traces \(N_\alpha\) with
\(N_\alpha/c(\alpha)\to\varrho\in(0,\infty)\).
Put \(V_y(r)=\nu(B(y,r))\).
For every fixed \(\sigma\in\mathcal P_p(X)\),
\[
\begin{split}
    &\liminf_{\alpha\to\infty}
          W_p(\overline\mu_{\gamma_\alpha}^\alpha,\sigma)\\
    &\qquad\geq\sup_{y\in X}
       \left[
          \left(\frac1\varrho\int_0^\varrho
                         (V_y^\dagger(s))^p\,ds\right)^{1/p}
          -M_p(\sigma;y)
       \right]_+.
\end{split}
\]
This bound requires neither averaged kernel localization nor
convergence of the state measures.
\end{corollary}

\begin{proof}
For a fixed \(y\), the scalar filling duality and vague convergence
against \((t-d(\cdot,y)^p)_+\) give
\[
    \liminf_{\alpha\to\infty}\frac1{c(\alpha)}
       \int_Xd(x,y)^p h_{\gamma_\alpha}^\alpha(x)\,d\nualpha(x)
       \geq\int_0^\varrho(V_y^\dagger(s))^p\,ds.
\]
Divide by \(N_\alpha/c(\alpha)\to\varrho\), take \(p\)-th roots,
and use the triangle inequality with \(\delta_y\).
The resulting lower bound holds for every fixed \(y\), so one may
take its supremum.
\end{proof}

\begin{remark}
\label{rem:bessel-uncertainty}
The occupation argument also applies to normalized families satisfying
\[
    \sum_{j=1}^N|\langle f,\varphi_j\rangle_\alpha|^2
       \leq B\|f\|_\alpha^2.
\]
Their occupation function is at most \(B\), so scalar filling gives
\[
    \sum_{j=1}^Nq_{\alpha,a}(\varphi_j)
       \geq B\int_0^{N/B}
          \inf\{t\geq0:\nu_\alpha(\{a\leq t\})\geq s\}\,ds.
\]
Under a common integrable envelope, the corresponding bounds are
\[
    N\leq\int_X\min\{B,N\Phi_\alpha\}\,d\nualpha,\qquad
    N(1-\delta)\leq B K_{\alpha,\Phi_\alpha}(\delta).
\]
These follow by applying the preceding proofs to \(s_\alpha/B\);
equivalently, the operator
\(B^{-1}\sum_j\langle\,\cdot\,,\varphi_j\rangle_\alpha\varphi_j\)
is a positive contraction.  Thus a finite Bessel bound gives
quantitative control when exact orthogonality is weakened.
\end{remark}

The examples in the next section identify the frame volume and the relevant
moment costs in concrete spaces.

% -----------------------------------------------------------------------------
% Concrete applications
% -----------------------------------------------------------------------------

\section{Applications}
\label{sec:applications}

This section records the concrete statements supplied by the abstract
theorems.  The point is not merely that the models fit the hypotheses:
each one gives a quantitative obstruction to placing many orthogonal
states where the indicated cost is small.  In Euclidean phase space the
obstruction has the power scale \(N^{1+p/(2d)}\); in the hyperbolic ball
it has the scale \(N(\log N)^p\).  In the varying-measure examples it is
a limiting filling law, determined by the limiting density of available
states.

To make the consequences easy to locate, every subsection begins by
stating its conclusion before verifying the corresponding frame and
localization assumptions.  The verification paragraphs also identify
which conclusions are exact, which are asymptotically sharp, and which
require the additional moment hypothesis.

We write \(\omega_\ell=\pi^{\ell/2}/\Gamma(1+\ell/2)\) for the
volume of the Euclidean unit ball in \(\R^\ell\).  When a model has no
varying parameter, we suppress \(\alpha\).
For sets in Euclidean space, a van Hove family \(A_R\subset\R^\ell\)
satisfies
\[
    |A_R\mathbin{\triangle}(A_R-u)|=o(|A_R|)
    \qquad\text{for every fixed }u\in\R^\ell.
\]
Dilates of a bounded Lipschitz domain are examples.

\subsection{Bargmann--Fock space: the exact radial spectrum}
\label{subsec:application-fock}

The conclusion here is a sharp Euclidean power law, together with the
full radial spectrum: \(N\) orthogonal Fock states have quadratic or
higher radial cost at least of order
\(\alpha^{-p/2}N^{1+p/(2d)}\), and that order cannot be improved.

For \(\alpha>0\), let \(\mathcal F_\alpha^2(\C^d)\) be the space of
entire functions with norm
\[
    \|f\|_\alpha^2
       =\left(\frac{\alpha}{\pi}\right)^d
          \int_{\C^d}|f(z)|^2e^{-\alpha|z|^2}\,dV(z).
\]
Its reproducing kernel and normalized frame vectors are
\[
    K_\alpha(w,z)=e^{\alpha w\cdot\overline z},\qquad
    k_z^\alpha(w)=e^{\alpha w\cdot\overline z-\alpha|z|^2/2}.
\]
Consequently,
\begin{equation}\label{eq:fock-frame-data}
\begin{gathered}
    d\nualpha(z)=\left(\frac{\alpha}{\pi}\right)^d dV(z),
    \qquad c(\alpha)=\left(\frac{\alpha}{\pi}\right)^d,\qquad d\nu=dV,\\
    Q_\alpha(z,dw)=\left(\frac{\alpha}{\pi}\right)^d
                       e^{-\alpha|z-w|^2}\,dV(w),\qquad
    V_\alpha(r)=\frac{\alpha^d}{d!}r^{2d}.
\end{gathered}
\end{equation}
The phase-space probability measure of a unit vector is
\[
    d\mu_f^\alpha(z)
       =\left(\frac{\alpha}{\pi}\right)^d
                    |f(z)|^2e^{-\alpha|z|^2}\,dV(z).
\]
For the radial cost \(a(z)=|z|^p\), the corresponding filling inverse is
\[
    F_{\alpha,|z|^p}^{\dagger}(s)
       =(d!)^{p/(2d)}\alpha^{-p/2}s^{p/(2d)}.
\]
Here \(\mathcal E_{\alpha,|z|^p}(\gamma)\) and
\(S_{\alpha,|z|^p}(N)\) denote, respectively, the energy and
interpolated eigenvalue sum introduced above.

\begin{proposition}
\label{prop:fock-application}
For \(p>0\), let \(0\leq\gamma\leq I\) be trace class, with
\(N=\Tr\gamma\).  Then
\begin{equation}\label{eq:fock-filling-bound}
\begin{aligned}
    \mathcal E_{\alpha,|z|^p}(\gamma)
       &\geq S_{\alpha,|z|^p}(N)\\
       &\geq\int_0^N F_{\alpha,|z|^p}^{\dagger}(s)\,ds\\
       &=\frac{2d}{2d+p}\,(d!)^{p/(2d)}
          \alpha^{-p/2}N^{1+p/(2d)}.
\end{aligned}
\end{equation}
The monomials
\[
    e_\beta^\alpha(z)
       =\left(\frac{\alpha^{|\beta|}}{\beta!}\right)^{1/2}z^\beta,
       \qquad \beta\in\{0,1,\ldots\}^d,
\]
form an orthonormal eigenbasis of \(T_{|z|^p}^\alpha\), with
\begin{equation}\label{eq:fock-radial-eigenvalues}
    \lambda_{\beta,\alpha}(|z|^p)
       =\alpha^{-p/2}
          \frac{\Gamma(|\beta|+d+p/2)}{\Gamma(|\beta|+d)}.
\end{equation}
The eigenvalue at degree \(m\) has multiplicity
\(\binom{m+d-1}{d-1}\).  At fixed \(\alpha\), the last expression in
\eqref{eq:fock-filling-bound} is asymptotic to
\(\sum_{j=1}^N\lambda_{j,\alpha}(|z|^p)\) as \(N\to\infty\).
\end{proposition}

\begin{proof}
The scalar filling integral obtained from \eqref{eq:fock-frame-data} is
\[
    \int_0^N\bigl(V_\alpha^\dagger(s)\bigr)^p\,ds
       =\alpha^{-p/2}(d!)^{p/(2d)}
                       \int_0^Ns^{p/(2d)}\,ds,
\]
which proves the lower bound.  Invariance under coordinate rotations
makes the radial Toeplitz form diagonal in the monomials.  The ratio
of its radial integrals is
\[
    \frac{\int_0^\infty r^{2|\beta|+2d+p-1}e^{-\alpha r^2}\,dr}
         {\int_0^\infty r^{2|\beta|+2d-1}e^{-\alpha r^2}\,dr},
\]
giving \eqref{eq:fock-radial-eigenvalues}.  All these integrals are
finite, so the polynomials give a dense form domain.  Finally,
\begin{equation}\label{eq:fock-kernel-moments}
    \int_{\C^d}|z-w|^p\,Q_\alpha(z,dw)
       =\alpha^{-p/2}\frac{\Gamma(d+p/2)}{\Gamma(d)}.
\end{equation}
Corollary~\ref{cor:power-cost-sandwich} proves the asserted asymptotic.
\end{proof}

For an orthonormal family and \(p\geq1\),
Theorem~\ref{thm:mean-dispersion} therefore gives the explicit estimate
\[
    \sum_{j=1}^N
       \bigl[\Delta_{\alpha,p}(\varphi_j)
                     +|\mu_{\alpha,p}(\varphi_j)|\bigr]^p
       \geq\frac{2d}{2d+p}(d!)^{p/(2d)}
                    \alpha^{-p/2}N^{1+p/(2d)}.
\]
For \(p=2\), the Euclidean variance identity replaces the expression
on the left by
\(\sum_j[\Delta_{\alpha,2}(\varphi_j)^2+
|\mu_{\alpha,2}(\varphi_j)|^2]\).
In this case \eqref{eq:fock-radial-eigenvalues} reads
\[
    T_{|z|^2}^\alpha e_\beta^\alpha
       =\frac{|\beta|+d}{\alpha}e_\beta^\alpha.
\]
Thus the oscillator interpretation is exact at the level of its
number operator and the indicated additive constant.

The Gaussian kernel also verifies all the semiclassical hypotheses
for \(a(z)=|z|^p\).  Its mass outside any fixed-radius ball tends to
zero uniformly in the center.  On a compact set of centers, a moment
of order strictly larger than \(p\), together with
\eqref{eq:fock-kernel-moments}, controls the tails in
\eqref{eq:kernel-moment-tail}.  Hence
Theorem~\ref{thm:semiclassical-filling} applies for every \(p>0\).

The same Gaussian overlap gives an effective-dimension law for large
Euclidean sets.  If \(A_R\subset\C^d\) is a van Hove family, then,
for every fixed \(0<\eps<1\),
\[
    \#\left\{j:\lambda_j\bigl(T_{\mathbf1_{A_R}}^\alpha\bigr)
                         \geq1-\eps\right\}
       =\left(\frac{\alpha}{\pi}\right)^d|A_R|+o(|A_R|)
       \qquad(R\to\infty).
\]
Indeed, \(Q_\alpha(z,dw)\) is a translation-invariant probability
kernel, so its leakage across \(A_R\) is \(o(|A_R|)\) by dominated
convergence.  This is Theorem~\ref{thm:frame-concentration} in the
Bargmann--Fock model.

\subsection{Weighted Bergman space: logarithmic cost in the ball}
\label{subsec:application-bergman}

Here the conclusion changes qualitatively: hyperbolic volume grows
exponentially, so the least total \(p\)-th radial cost of \(N\)
orthogonal states is of order \(N(\log N)^p\), not a Euclidean power.

For \(\alpha>-1\), equip the holomorphic functions on \(\B^d\) with
the norm of \(L^2(\mu_\alpha)\), where
\[
    d\mu_\alpha(z)=C_{\alpha,d}(1-|z|^2)^\alpha\,dV(z),
    \qquad C_{\alpha,d}
       =\frac{\Gamma(\alpha+d+1)}{\pi^d\Gamma(\alpha+1)}.
\]
The reproducing kernel is
\[
    K_\alpha(z,w)=(1-z\cdot\overline w)^{-(\alpha+d+1)},
\]
so the frame measure is
\begin{equation}\label{eq:bergman-frame-data}
    d\nualpha(z)=C_{\alpha,d}\frac{dV(z)}{(1-|z|^2)^{d+1}},
    \qquad c(\alpha)=C_{\alpha,d},\qquad
    d\nu(z)=\frac{dV(z)}{(1-|z|^2)^{d+1}}.
\end{equation}
Use the invariant distance with normalization
\[
    d_{\B}(z,w)=\operatorname{arctanh}|\phi_w(z)|,\qquad
    d_{\B}(0,z)=\operatorname{arctanh}|z|,
\]
where \(\phi_w\) is a ball automorphism interchanging \(0\) and \(w\).
This specifies the metric scale in all constants below.
Polar coordinates give
\begin{equation}\label{eq:bergman-ball-volume}
\begin{gathered}
    V_\alpha(r)=A_{\alpha,d}\sinh^{2d}r,\qquad
    A_{\alpha,d}=\frac{\Gamma(\alpha+d+1)}{d!\Gamma(\alpha+1)},\\
    V_\alpha^\dagger(s)
       =\operatorname{arsinh}\!
          \left((s/A_{\alpha,d})^{1/(2d)}\right).
\end{gathered}
\end{equation}
Indeed, the substitution \(t=|z|^2\) reduces the volume calculation to
the antiderivative
\(d^{-1}[t/(1-t)]^d\).

\begin{proposition}
\label{prop:bergman-application}
Put \(r_0(z)=d_{\B}(0,z)\).  For \(p>0\), write
\(N=n+\theta\), where \(n\in\mathbb N\cup\{0\}\) and
\(0\leq\theta<1\).  Then
\begin{equation}\label{eq:bergman-filling}
    \sum_{j=1}^n\lambda_{j,\alpha}(r_0^p)
       +\theta\lambda_{n+1,\alpha}(r_0^p)
       \geq\int_0^N
          \left[
            \operatorname{arsinh}\!
                 \left((s/A_{\alpha,d})^{1/(2d)}\right)
          \right]^p ds.
\end{equation}
At each fixed \(\alpha>-1\), this bound is asymptotically sharp:
\begin{equation}\label{eq:bergman-large-family}
    \sum_{j=1}^N\lambda_{j,\alpha}(r_0^p)
       \sim \frac{N(\log N)^p}{(2d)^p}
       \qquad(N\to\infty).
\end{equation}
The weak-measure, averaged-localization, and kernel-tail hypotheses
hold as \(\alpha\to\infty\) for these same symbols.
\end{proposition}

\begin{proof}
The lower bound follows from \eqref{eq:bergman-ball-volume}.
Both the frame measure and the absolute kernel overlap are invariant
under ball automorphisms.  At the origin \(k_0^\alpha=1\), and
\(Q_\alpha(0,dw)=d\mu_\alpha(w)\).  Thus
\[
    \eps_{\alpha,p}^p
       =\int_{\B^d}
           \bigl(\operatorname{arctanh}|w|\bigr)^p\,d\mu_\alpha(w)
       <\infty.
\]
Finiteness follows from the integrability of every power of
\(|\log(1-|w|)|\) against \((1-|w|^2)^\alpha dV\).
The same observation places all holomorphic polynomials in the form
domain.  Corollary~\ref{cor:power-cost-sandwich} now gives asymptotic
equality in \eqref{eq:bergman-filling}.  Since
\[
    \operatorname{arsinh}\!
           \left((s/A_{\alpha,d})^{1/(2d)}\right)
       =\frac{\log s}{2d}+O_{\alpha,d}(1),
       \qquad s\to\infty,
\]
integration proves \eqref{eq:bergman-large-family}.

For the semiclassical assertions, under \(\mu_\alpha\) the variable
\(|w|^2\) has the beta distribution with parameters \(d,\alpha+1\).
It tends to zero in probability as \(\alpha\to\infty\).
For \(\alpha\geq\alpha_0>-1\) these distributions are stochastically
dominated by the distribution at \(\alpha_0\).  Their hyperbolic
moments of every fixed order are consequently uniformly integrable
and tend to zero.  Invariance gives localization uniformly in the
center.  The triangle inequality
\(r_0(w)\leq r_0(z)+d_{\B}(z,w)\), followed by a higher moment bound,
gives \eqref{eq:kernel-moment-tail} on compact sets of centers.
The measure convergence is already exact in
\eqref{eq:bergman-frame-data}.
\end{proof}

In particular, for \(p=1\) every orthonormal family satisfies
\[
    \sum_{j=1}^N
       \bigl[\Delta_\alpha(\varphi_j)
                     +d_{\B}(0,\mu_\alpha(\varphi_j))\bigr]
       \geq \int_0^N
          \operatorname{arsinh}\!
                 \left((s/A_{\alpha,d})^{1/(2d)}\right)ds.
\]
The \(N\log N\) scale comes from exponential growth of hyperbolic
volume.  It is the counterpart of the power law in Fock space.

\subsection{Gabor frames and anisotropic phase-space costs}
\label{subsec:application-gabor}

The conclusion is a collective time--frequency uncertainty principle:
an orthonormal family of \(N\) signals cannot all have small
spectrogram moments.  In particular, the lower bounds below are
independent of the window; a moment assumption on the window is needed
only for their asymptotic sharpness.

We use the Fourier convention
\(\widehat f(\xi)=\int_{\R^d}f(t)e^{-2\pi i t\cdot\xi}\,dt\).
Fix \(g\in L^2(\R^d)\) with \(\|g\|_2=1\), and put
\[
    k_{(x,\xi)}=M_\xi T_xg,\qquad
    V_gf(x,\xi)=\langle f,M_\xi T_xg\rangle,
\]
where \(T_xg(t)=g(t-x)\) and
\(M_\xi g(t)=e^{2\pi i\xi\cdot t}g(t)\).
Plancherel and Fubini give Moyal's identity
\[
    \int_{\R^{2d}}|V_gf(z)|^2\,dz=\|f\|_2^2.
\]
Thus \(d\nu(z)=dz\) is the normalized frame measure.  Equivalently,
the isometric image \(V_gL^2(\R^d)\) is an RKHS to which the preceding
results apply.  Its phase-space measures are
\(|V_gf(z)|^2dz\).

\paragraph{Radial and anisotropic consequences.}
For every orthonormal family \(f_1,\ldots,f_N\) and \(p>0\),
\begin{equation}\label{eq:gabor-occupation}
    \sum_{j=1}^N\int_{\R^{2d}}|z|^p|V_gf_j(z)|^2\,dz
       \geq \frac{2d}{2d+p}\omega_{2d}^{-p/(2d)}
                         N^{1+p/(2d)}.
\end{equation}
This follows directly from \(V(r)=\omega_{2d}r^{2d}\), and requires
no decay assumption on \(g\).  If a moment on the left is infinite,
the inequality retains its extended-value meaning.
If
\[
    \int_{\R^{2d}}|z|^p|V_gg(z)|^2\,dz<\infty,
\]
the overlap kernel has a finite uniform \(p\)-th moment.  Its coherent
vectors belong to the form domain and span a dense subspace.
Corollary~\ref{cor:power-cost-sandwich} then makes the lower bound
asymptotically sharp for the radial spectral sums.
For the normalized Gaussian
\(g(t)=2^{d/4}e^{-\pi|t|^2}\), the Bargmann transform identifies this
model with \(\mathcal F_\pi^2(\C^d)\), up to an isometric reflection
of the phase-space coordinates.  Formula
\eqref{eq:fock-radial-eigenvalues} gives its exact radial spectrum.

The general concentration theorem gives a time--frequency
effective-dimension law.  If \(A_R\subset\R^{2d}\) is a van Hove family,
then, for every fixed
\(0<\eps<1\),
\[
    \#\left\{j:\lambda_j\bigl(T_{\mathbf1_{A_R}}\bigr)
                         \geq1-\eps\right\}
       =|A_R|+o(|A_R|)
       \qquad(R\to\infty).
\]
Indeed, the overlap kernel is a translate of
\(|V_gg|^2\,dz\), which is an integrable probability density by Moyal's
identity.  The leakage condition in
Theorem~\ref{thm:frame-concentration} therefore follows by dominated
convergence.  No moment condition on \(g\) is needed for this count.

The filling argument also allows costs that are not powers of a
metric.  For
\[
    a(x,\xi)=|x|^p+|\xi|^q,\qquad p,q>0,
\]
put
\[
    \kappa=\frac d p+\frac d q,\qquad
    C_{p,q,d}
       =\omega_d^2
          \frac{\Gamma(1+d/p)\Gamma(1+d/q)}
               {\Gamma(1+d/p+d/q)}.
\]
Integration first in \(\xi\), and then in the radial variable \(x\),
gives
\(\nu(\{a\leq t\})=C_{p,q,d}t^\kappa\).  Hence
\begin{equation}\label{eq:gabor-anisotropic}
    \sum_{j=1}^N\int_{\R^{2d}}a(z)|V_gf_j(z)|^2\,dz
       \geq\frac{\kappa}{\kappa+1}
                  C_{p,q,d}^{-1/\kappa}N^{1+1/\kappa}.
\end{equation}
For \(p=q=2\), a centered window with finite second moments satisfies
\[
\begin{split}
    \int_{\R^{2d}}(|x|^2+|\xi|^2)|V_gf(x,\xi)|^2\,dx\,d\xi
       ={}&\int_{\R^d}|t|^2|f(t)|^2\,dt
           +\int_{\R^d}|\xi|^2|\widehat f(\xi)|^2\,d\xi\\
          &+\int_{\R^d}|t|^2|g(t)|^2\,dt
           +\int_{\R^d}|\xi|^2|\widehat g(\xi)|^2\,d\xi
\end{split}
\]
for \(\|f\|_2=1\); centered means that the position and frequency means
of \(g\) vanish.  This identity follows from the two marginals of the
spectrogram.

The Gaussian window makes the comparison with the \(p=2\) result of
Malinnikova \cite[Theorem~3]{Malinnikova} completely explicit.  Set
\[
    g_0(t)=2^{d/4}e^{-\pi|t|^2}.
\]
Then
\[
    \int_{\R^d}|t|^2|g_0(t)|^2\,dt
       +\int_{\R^d}|\xi|^2|\widehat g_0(\xi)|^2\,d\xi
       =\frac{d}{2\pi}.
\]
For \(N\geq1\), let \(m\) be determined by
\[
    \binom{m+d-1}{d}<N\leq\binom{m+d}{d},
    \qquad r=N-\binom{m+d-1}{d}.
\]
Every orthonormal family \(f_1,\ldots,f_N\) in \(L^2(\R^d)\) then
satisfies
\[
\begin{aligned}
 &\sum_{j=1}^N\left(
       \int_{\R^d}|t|^2|f_j(t)|^2\,dt
       +\int_{\R^d}|\xi|^2|\widehat f_j(\xi)|^2\,d\xi\right)\\
 &\qquad\geq\frac{1}{\pi}\left[
       d\binom{m+d-1}{d+1}
       +\frac d2\binom{m+d-1}{d}
       +r\left(m+\frac d2\right)\right].
\end{aligned}
\]
Indeed, the Bargmann identification turns the phase-space quadratic
cost into \(T_{|z|^2}^{\pi}\), with eigenvalues
\((|\beta|+d)/\pi\).  The preceding identity subtracts
\(d/(2\pi)\) from each occupied level.  Thus the displayed right
side is the exact sum of the first \(N\) oscillator eigenvalues, and
equality is attained by filling those levels.  Malinnikova's theorem
gives the left-hand side above the lower bound
\(C_dN^{1+1/d}\) for some \(C_d>0\).  The present formula is its
sharp finite-\(N\) version: it has leading term
\[
    \frac{d}{d+1}\frac{(d!)^{1/d}}{\pi}N^{1+1/d},
\]
and makes both the coefficient and the finite-level correction
explicit.  When \(d=1\), it reduces to \(N^2/(2\pi)\).

The Riesz-basis obstructions in
\cite{GrochenigMalinnikova} use additional information about
completeness and localization; they are not consequences of a bound
on the number of states in a fixed region alone.

Thus \eqref{eq:gabor-occupation} and
\eqref{eq:gabor-anisotropic} are the promised results in this model:
they give explicit lower bounds for the sum of the time--frequency
costs of \emph{every} orthonormal family.  The Gaussian window further
turns the radial statement into the exact Fock spectrum, while a
finite ambiguity-function moment turns the radial lower bound into a
sharp asymptotic formula for the lowest \(N\) eigenvalues.

A semiclassical family is obtained without changing the underlying
Hilbert space.  For \(\alpha>0\), set
\[
    k_{(x,\xi)}^\alpha(t)
       =\alpha^{d/4}e^{2\pi i\alpha\xi\cdot t}
                             g\bigl(\sqrt\alpha(t-x)\bigr).
\]
Then
\[
\begin{gathered}
    d\nualpha(z)=\alpha^d\,dz,\qquad c(\alpha)=\alpha^d,\\
    Q_\alpha(z,dw)
       =\alpha^d
          |V_gg(\sqrt\alpha(w-z))|^2\,dw.
\end{gathered}
\]
Since \(|V_gg|^2\) is an integrable probability density, these kernels
localize.  If its \(p\)-th moment is finite, scaling and truncating
that integrable moment verify \eqref{eq:kernel-moment-tail} for
\(a(z)=|z|^p\).  Thus the weak-measure formulation recovers the usual
semiclassical scaling of Gabor localization operators.

\subsection{Paley--Wiener spaces and the role of bandwidth}
\label{subsec:application-paley-wiener}

The conclusion is a bandwidth-sensitive uncertainty principle.  A
band-limited orthonormal family has total spatial \(p\)-moment at least
the explicit quantity in \eqref{eq:pw-moment-bound}; equivalently, a
region \(A\) can support at most \(|\Omega||A|/(1-\eps)\) nearly
concentrated orthogonal signals.  The associated concentration operator
also recovers the leading time--bandwidth dimension count.

Let \(\Omega\subset\R^d\) be bounded and measurable, with
\(0<|\Omega|<\infty\), and let
\[
    PW_\Omega
       =\{f\in L^2(\R^d):\supp\widehat f\subset\Omega
                                      \text{ up to null sets}\}.
\]
We use the continuous representatives supplied by Fourier inversion.
The reproducing kernel and frame measure are
\[
\begin{gathered}
    K_\Omega(x,y)
       =\int_\Omega e^{2\pi i(x-y)\cdot\xi}\,d\xi,\qquad
    d\nu_\Omega(x)=|\Omega|\,dx,\\
    Q_\Omega(x,dy)
       =\frac{|K_\Omega(x,y)|^2}{|\Omega|}\,dy,\qquad
    d\mu_f(x)=|f(x)|^2\,dx.
\end{gathered}
\]
Writing
\[
    q_\Omega(u)=\frac{|\check{\mathbf1}_\Omega(u)|^2}{|\Omega|},
    \qquad \int_{\R^d}q_\Omega(u)\,du=1,
\]
we have \(Q_\Omega(x,dy)=q_\Omega(y-x)\,dy\).  The normalization is
Plancherel's theorem.
The actual bandwidth appears in the volume
\(V_\Omega(r)=|\Omega|\omega_dr^d\).  For every orthonormal family,
\begin{equation}\label{eq:pw-moment-bound}
    \sum_{j=1}^N\int_{\R^d}|x|^p|f_j(x)|^2\,dx
       \geq\frac d{d+p}(|\Omega|\omega_d)^{-p/d}N^{1+p/d},
       \qquad p>0.
\end{equation}
The corresponding center-dispersion bound for \(p\geq1\) follows
from Theorem~\ref{thm:mean-dispersion}.
For a measurable set \(A\) of finite volume, let
\[
    C_{\Omega,A}=P_\Omega\mathbf1_AP_\Omega\big|_{PW_\Omega},
\]
where \(P_\Omega\) is the orthogonal projection onto \(PW_\Omega\).
Theorem~\ref{thm:frame-concentration} gives
\[
\begin{aligned}
 |\Omega||A|
    -\frac1\eps\int_A\int_{\R^d\setminus A}
                         |K_\Omega(x,y)|^2\,dy\,dx
 &\leq\#\{j:\lambda_j(C_{\Omega,A})\geq1-\eps\}\\
 &\leq |\Omega||A|
    +\frac1{1-\eps}\int_A\int_{\R^d\setminus A}
                         |K_\Omega(x,y)|^2\,dy\,dx.
\end{aligned}
\]
Thus, if each \(f_j\) has at least \(1-\eps\) of its squared norm in
\(A\), then
\[
    N(1-\eps)\leq|\Omega|\,|A|.
\]
For a van Hove family \(A_R\),
\[
 \frac{1}{|\Omega||A_R|}
    \int_{A_R}\int_{\R^d\setminus A_R}|K_\Omega(x,y)|^2\,dy\,dx
   =\int_{\R^d}q_\Omega(u)
       \frac{|A_R\setminus(A_R-u)|}{|A_R|}\,du
   \longrightarrow0.
\]
Indeed, this is dominated convergence and the defining translation
property of \(A_R\).  Hence
\[
    \#\{j:\lambda_j(C_{\Omega,A_R})\geq1-\eps\}
       =|\Omega||A_R|+o(|A_R|)
       \qquad(R\to\infty).
\]
This is the Landau--Slepian--Pollak effective-dimension statement in
the present normalization \cite{LandauPollak}.  For \(\Omega=[-W,W]\)
and an interval of length \(T\), its leading term is \(2WT\).

For a rectangular band and \(0<p<1\), the moment lower bound is also
sharp as \(N\to\infty\); for \(p\geq1\), the lower bound and the
time--bandwidth capacity remain valid, but the sinc tails prevent this
particular kernel-moment argument from proving a matching upper bound.

For \(\Omega_\alpha=\alpha\Omega\), take
\(c(\alpha)=|\Omega_\alpha|\) and \(d\nu=dx\).  With
the probability density \(q_\Omega\) just introduced,
the overlap kernels are
\[
    Q_\alpha(x,dy)=\alpha^d q_\Omega(\alpha(x-y))\,dy.
\]
Their mass outside any fixed-radius ball tends to zero.  Thus the
bounded-symbol trace results apply for every such \(\Omega\).
For unbounded moment symbols, an additional check is essential.

Suppose, for example,
\(\Omega=\prod_{\ell=1}^d[-a_\ell/2,a_\ell/2]\), with \(a_\ell>0\).
Writing \(\operatorname{sinc}t=\sin(\pi t)/(\pi t)\), its Shannon
basis is
\[
    \varphi_{\boldsymbol n}(x)
       =\prod_{\ell=1}^d
          \sqrt{a_\ell}\,
          \operatorname{sinc}(a_\ell x_\ell-n_\ell),
       \qquad \boldsymbol n\in\mathbb Z^d.
\]
The squared sinc tail has finite \(p\)-th moment exactly when
\(0<p<1\).  The same threshold holds for the radial moments of \(q\):
the upper bound follows from
\(|u|^p\leq\sum_\ell|u_\ell|^p\) for \(p<1\), and divergence for
\(p\geq1\) follows by restricting the other coordinates to a bounded
set of positive measure.  Thus these basis vectors have infinite
first moments.  They cannot be used as examples with bounded
first-moment dispersion.

For \(0<p<1\), the finite kernel moment yields both the fixed-bandwidth
asymptotic
\[
    \sum_{j=1}^N\lambda_{j,\Omega}(|x|^p)
       \sim\frac d{d+p}(|\Omega|\omega_d)^{-p/d}N^{1+p/d}
       \qquad(N\to\infty)
\]
and the semiclassical filling law for \(\Omega_\alpha\).
The form domains are dense: inverse Fourier transforms of smooth
functions compactly supported in the interior of the rectangle give
a dense Schwartz subspace.  This density also holds for all \(p>0\);
what fails for \(p\geq1\) is the kernel moment hypothesis, since
\(\widetilde{|x|^p}^{\,\alpha}(x)=\infty\).
The exact lower bound \eqref{eq:pw-moment-bound} remains valid.

\subsection{Orthogonal polynomials and determinantal ensembles}
\label{subsec:application-oprl}

Let \(\mu\) be a finite positive measure with infinite compact support
\(E\subset\R\), and let \(p_0,p_1,\ldots\) be its orthonormal
polynomials, with positive leading coefficients.  Put
\[
    \mathcal H_n=\operatorname{span}\{p_0,\ldots,p_{n-1}\},
    \qquad K_n(x,y)=\sum_{j=0}^{n-1}p_j(x)p_j(y).
\]
The normalized kernel vectors give the frame data
\[
    d\nu_n(x)=K_n(x,x)\,d\mu(x),\qquad
    c(n)=n,\qquad dm_n(x)=\frac1nK_n(x,x)\,d\mu(x).
\]
In particular \(m_n\) is a probability measure.  The diagonal is
strictly positive because \(p_0\) is a nonzero constant.
For a bounded symbol \(\sigma\),
\[
    T_\sigma^n=P_nM_\sigma P_n\big|_{\mathcal H_n},
\]
where \(P_n\) is the orthogonal projection in \(L^2(\mu)\) onto
\(\mathcal H_n\).

\paragraph{Filling consequence.}
Whenever \(m_n\rightharpoonup\nu\), the first \(N_n\sim\varrho n\)
orthogonal-polynomial modes fill the lowest \(\nu\)-volume fraction of
the cost:
\[
    \frac1n\sum_{j=1}^{N_n}\lambda_{j,n}(a)
       \longrightarrow\int_0^\varrho F_a^\dagger(s)\,ds,
    \qquad F_a(t)=\nu(\{a\leq t\}).
\]
For a regular measure on \([-1,1]\) and \(a(x)=|x|^p\), this is the
explicit arcsine law displayed in \eqref{eq:oprl-filling-limit}.  The
proposition below supplies the localization input; it is the mechanism
that turns weak convergence of Christoffel measures into this spectral
filling statement.

\begin{proposition}
\label{prop:oprl-localization}
Write the three-term recurrence as
\[
    xp_j=a_{j+1}p_{j+1}+b_jp_j+a_jp_{j-1},\qquad a_0=0.
\]
Then for every \(\eps>0\),
\begin{equation}\label{eq:oprl-localization-bound}
    \int_E Q_n(x,\{y:|x-y|\geq\eps\})\,dm_n(x)
       \leq\frac{2a_n^2}{n\eps^2}.
\end{equation}
Consequently, along every subsequence on which \(m_n\) converges
weakly, both hypotheses
\eqref{eq:measure-convergence}--\eqref{eq:averaged-localization}
hold.  The same estimate holds for measures \(\mu_n\) varying with
\(n\), provided their moments exist and their supports are infinite;
averaged localization follows whenever \((a_n^{(n)})^2/n\to0\).
\end{proposition}

\begin{proof}
The Christoffel--Darboux identity gives
\[
    (x-y)K_n(x,y)
       =a_n\bigl[p_n(x)p_{n-1}(y)-p_{n-1}(x)p_n(y)\bigr].
\]
Orthogonality therefore yields
\begin{equation}\label{eq:oprl-overlap-second-moment}
    \iint (x-y)^2|K_n(x,y)|^2\,d\mu(x)d\mu(y)=2a_n^2.
\end{equation}
Since
\[
    dm_n(x)Q_n(x,dy)
       =\frac1n|K_n(x,y)|^2\,d\mu(x)d\mu(y),
\]
Chebyshev's inequality proves \eqref{eq:oprl-localization-bound}.
For fixed compact support, \(a_n\leq\|M_x\|<\infty\).
The calculation is unchanged when the orthogonal polynomials and the
measure depend on \(n\).
\end{proof}

Simon's theorem \cite{SimonCD} identifies the subsequential weak
limits of \(m_n\) with those of the normalized zero-counting measures
of \(p_n\).  If \(E\) has positive logarithmic capacity and \(\mu\)
is regular in the sense of Stahl--Totik, these measures converge to
the equilibrium measure of \(E\).  Thus one needs a limiting
Christoffel measure, rather than a pointwise limit of the kernel
diagonal.

Suppose \(m_n\rightharpoonup\nu\).  For every continuous \(a:E\to[0,\infty)\),
all form and tail conditions are automatic.  If \(N_n/n\to\varrho\)
with \(0<\varrho<1\), Theorem~\ref{thm:semiclassical-filling} gives
\begin{equation}\label{eq:oprl-filling-limit}
    \frac1n\sum_{j=1}^{N_n}\lambda_{j,n}(a)
       \longrightarrow\int_0^\varrho F_a^\dagger(s)\,ds,
    \qquad F_a(t)=\nu(\{a\leq t\}).
\end{equation}
For example, \(a(x)=|x-e|^p\) gives the collective price of
concentrating orthogonal polynomials near \(e\).
When \(\mu\) is regular on \([-1,1]\) and \(e=0\), the limiting
measure is \(d\nu(x)=dx/(\pi\sqrt{1-x^2})\), so
\[
    \frac1n\sum_{j=1}^{N_n}\lambda_{j,n}(|x|^p)
       \longrightarrow
          \int_0^\varrho\sin^p\left(\frac{\pi s}{2}\right)\,ds.
\]
The varying frame measure here is essential: it is
\(K_n(x,x)d\mu(x)\), not a prescribed scalar multiple of a fixed
measure.

\subsubsection*{The same kernel as a particle process}

The same conclusion has a probabilistic form.  The expected number of
particles in \(A\) is exactly its capacity for orthogonal states, as
made precise by \eqref{eq:oprl-particle-capacity}; normalized linear
statistics converge to their equilibrium average.

The probability law
\begin{equation}\label{eq:oprl-determinantal-law}
    \frac1{n!}\det[K_n(x_i,x_j)]_{i,j=1}^n
                  \prod_{i=1}^n d\mu(x_i)
\end{equation}
is the orthogonal-polynomial determinantal ensemble.  Its one-point
intensity is precisely \(\nu_n\):
\[
    \mathbb E\,\#\{i:x_i\in A\}=\nu_n(A).
\]
For a real bounded symbol \(\sigma\), put
\(L_n(\sigma)=\sum_{i=1}^n\sigma(x_i)\).
The one- and two-point correlation identities imply
\[
\begin{aligned}
    \operatorname{Var}L_n(\sigma)
       &=\Tr\bigl(T_{\sigma^2}^n-(T_\sigma^n)^2\bigr)\\
       &=\frac12\iint|\sigma(x)-\sigma(y)|^2
                         |K_n(x,y)|^2\,d\mu(x)d\mu(y).
\end{aligned}
\]
Indeed, the two-point density is
\[
    K_n(x,x)K_n(y,y)-|K_n(x,y)|^2.
\]
Expanding the variance and using
\(\int|K_n(x,y)|^2d\mu(y)=K_n(x,x)\) gives the second line.
This is exactly the trace defect from
\eqref{eq:dirichlet-trace-defect}.  In particular, if
\(\sigma\in C(E)\), then
\[
    \frac1n\operatorname{Var}L_n(\sigma)\longrightarrow0,\qquad
    \frac1nL_n(\sigma)\longrightarrow\int_E\sigma\,d\nu
                                      \quad\hbox{in }L^2.
\]
For Lipschitz \(\sigma\), \eqref{eq:oprl-overlap-second-moment}
also gives the finite-\(n\) estimate
\(\operatorname{Var}L_n(\sigma)\leq
a_n^2\Lip(\sigma)^2\).
The finer fluctuation results in \cite{BreuerDuits} concern a scale
below this first-order trace limit.

The occupation inequality now has a direct probabilistic reading.
If \(f_1,\ldots,f_N\in\mathcal H_n\) are orthonormal and each satisfies
\(\int_A|f_j|^2d\mu\geq1-\eps\), then
\begin{equation}\label{eq:oprl-particle-capacity}
    N(1-\eps)
       \leq\nu_n(A)=\mathbb E\,\#\{i:x_i\in A\}.
\end{equation}
Thus the expected particle count is also the capacity of \(A\) for
orthogonal states.  More generally,
Theorem~\ref{thm:frame-concentration} identifies the near-one spectrum
of the compression of \(\mathbf1_A\) to \(\mathcal H_n\): its effective
dimension is controlled by the same expected count and by the kernel
leakage across \(A\).  The displayed inequality is its
orthonormal-family consequence.

\subsubsection*{A worked varying-weight example: the GUE}

For the GUE, the preceding filling law becomes the concrete
semicircle formula \eqref{eq:gue-filling-limit}: a fraction
\(\varrho\) of the modes has minimum quadratic compressed energy
\(n\int_0^\varrho F_{\rm sc}^\dagger(s)\,ds+o(n)\).  This is a
statement about \(T_{x^2}^n\), rather than just about the zeros of a
Hermite polynomial.

Take
\[
    d\mu_n(x)=\sqrt{\frac n{2\pi}}e^{-nx^2/2}\,dx,
    \qquad \mathcal H_n=\{\text{polynomials of degree less than }n\}.
\]
The law \eqref{eq:oprl-determinantal-law}, with this varying weight,
is the eigenvalue law of the Gaussian unitary ensemble with
quadratic potential \(x^2/2\).  The recurrence coefficients are
\[
    a_j^{(n)}=\sqrt{j/n},\qquad b_j^{(n)}=0.
\]
Hence \(a_n^{(n)}=1\), and the localization error in
\eqref{eq:oprl-localization-bound} is at most \(2/(n\eps^2)\).
The normalized frame measures converge to
\[
    dm_n(x)=\frac1nK_n(x,x)\,d\mu_n(x)
       \rightharpoonup
       d\nu_{\rm sc}(x)
          =\frac1{2\pi}\sqrt{4-x^2}\,
                              \mathbf1_{[-2,2]}(x)\,dx.
\]
This normalization can be checked directly from the recurrence.
For each fixed \(k\), the diagonal entries of its \(2k\)-th power
are sums over closed nearest-neighbor paths.  Away from the lower
endpoint, each path starting at \(j\) has weight
\((j/n)^k+o(1)\), and there are \(\binom{2k}{k}\) such paths.
Averaging \(0\leq j<n\) therefore gives
\[
    \int x^{2k}\,dm_n(x)\longrightarrow
       \binom{2k}{k}\int_0^1t^k\,dt
       =\frac1{k+1}\binom{2k}{k};
\]
odd moments vanish.  The finitely many paths meeting the lower
endpoint have vanishing normalized contribution.  These moment
limits and the bounded higher moments imply weak convergence to
the uniquely determined semicircle measure.

Here an unbounded cost can also be handled explicitly.  The recurrence
gives
\[
    \int x^2\,dm_n(x)=1,\qquad
    \int x^4\,dm_n(x)=2+\frac1{n^2}.
\]
Using the marginal identity for the overlap measure, for every
compact \(K\subset\R\) and \(R>0\) we obtain
\[
    \int_K\int_{\{y^2>R\}}y^2Q_n(x,dy)\,dm_n(x)
       \leq\int_{\{y^2>R\}}y^2\,dm_n(y)
       \leq\frac3R.
\]
This proves \eqref{eq:kernel-moment-tail} for \(a(x)=x^2\).
All polynomials have finite Gaussian moments.  Consequently,
\begin{equation}\label{eq:gue-filling-limit}
    \frac1n\sum_{j=1}^{N_n}\lambda_{j,n}(x^2)
       \longrightarrow
          \int_0^\varrho F_{\rm sc}^\dagger(s)\,ds,\qquad
    F_{\rm sc}(t)=\nu_{\rm sc}([-\sqrt t,\sqrt t]),
\end{equation}
whenever \(N_n/n\to\varrho\in(0,1)\).
This is a concrete use of weak convergence with varying weights
and controlled energy tails.

One must distinguish the compressed cost from a function of the
compressed coordinate.  If \(J_n=T_x^n\) is the finite Jacobi matrix,
then
\[
    T_{x^2}^n
       =J_n^2+(a_n^{(n)})^2
                    (p_{n-1}\otimes p_{n-1}).
\]
Indeed, insert \(P_n+(I-P_n)\) between the two multiplication
operators; only \(xp_{n-1}\) has a component outside \(\mathcal H_n\).
The eigenvalues of \(J_n\) are the zeros of \(p_n\), whereas
\eqref{eq:gue-filling-limit} concerns \(T_{x^2}^n\).
In the GUE the correction has trace one, so it does not affect the
normalized filling limit.  These examples concern projection
determinantal ensembles, including the unitary (\(\beta=2\)) matrix
ensembles; no determinantal representation for general \(\beta\)
is being assumed.

\end{document}